\documentclass{siamart190516}
\usepackage{backref}
\usepackage{graphicx}
\graphicspath{ {Images/} }
\usepackage{booktabs}
\usepackage{graphicx}
\usepackage{amsmath}
\usepackage{amssymb,bm}
\usepackage{mathtools}
\usepackage{geometry}
\usepackage{cite}
\usepackage{algorithm}
\usepackage{algorithmic}
\usepackage{enumitem}   	
\usepackage{subcaption}

\usepackage{bbm}

\crefname{equation}{}{}
\crefname{theorem}{Theorem}{Theorems}
\crefname{lemma}{Lemma}{Lemmas}
\crefname{corollary}{Corollary}{Corollaries}
\crefname{proposition}{Proposition}{Propositions}
\crefname{remark}{Remark}{Remarks}
\crefname{algorithm}{Algorithm}{Algorithms}
\crefname{figure}{Figure}{Figures}
\crefname{table}{Table}{Tables}
\crefname{section}{Section}{Sections}
\crefname{subsection}{Section}{Sections}

\usepackage{todonotes}

\graphicspath{ {./images/} }

\newcommand{\bb}[1]{\mathbb{#1}}

\newcommand{\R}{\mathbb R}

\newcommand{\Var}{\mathrm{Var}}
\newcommand{\Cov}{\mathrm{Cov}}
\newcommand{\Kzero}{K_{0}}

\newtheorem{remark}[theorem]{Remark}

\newcommand{\mat}[1]{\mathbf{#1}}			
\renewcommand{\vec}[1]{\boldsymbol{#1}}		
\newcommand{\kryl}{\mathcal{K}}     
\newcommand{\orthproj}[1]{\mathbbmss{#1}^{\perp}}     
\newcommand{\proj}[1]{\mathbbmss{#1}}     
\newcommand{\sketch}{\Omega}			
\newcommand{\skproj}[1]{\mathbbmss{#1}^{\sketch}}     

\DeclareMathOperator*{\argmin}{argmin}		
\DeclareMathOperator{\diag}{diag}			
\DeclareMathOperator{\vspan}{span}			
\DeclareMathOperator{\range}{range}			
\DeclareMathOperator{\cond}{cond}			
\DeclareMathOperator{\err}{err}				
\DeclareMathOperator{\prob}{\mathbb{P}}		
\DeclareMathOperator{\expec}{\mathbb{E}}	
\DeclareMathOperator{\sign}{sign}			
\DeclarePairedDelimiter{\norm}{\lVert}{\rVert}		
\DeclarePairedDelimiter{\abs}{\lvert}{\rvert}		
\DeclarePairedDelimiter{\scpr}{\langle}{\rangle}		

\usepackage{textcomp}
\usepackage{stmaryrd}

\begin{document}

\title{Randomized biorthogonalization through a two-sided Gram-Schmidt process}

\author{Laura Grigori\thanks{Institute of Mathematics, EPFL, Lausanne, and PSI Center for Scientific Computing, Theory and Data, Villigen PSI, 
Switzerland}\and 
Lorenzo Piccinini\thanks{Dipartimento di Matematica, 
        Alma Mater Studiorum - Universit\`a di Bologna, Piazza di Porta San Donato 5,
        40126 Bologna, Italy. Affiliated to the Italian INdAM-GNCS (Gruppo Nazionale di Calcolo Scientifico).
}\and
Igor Simunec\thanks{Institute of Mathematics, EPFL, Lausanne, Switzerland}}

\maketitle

\begin{abstract}
    We propose and analyze a randomized two-sided Gram-Schmidt process for the biorthogonalization of two given matrices $\mat X, \mat Y \in\bb{R}^{n\times m}$. The algorithm aims to find two matrices $\mat Q, \mat P \in\bb{R}^{n\times m}$ such that ${\rm range}(\mat X) = {\rm range}(\mat Q)$, ${\rm range}(\mat Y) = {\rm range}(\mat P)$ and $(\sketch \mat Q)^T \sketch \mat P = \mat I$, where $\sketch \in\bb{R}^{s \times n}$ is a sketching matrix satisfying an oblivious subspace $\varepsilon$-embedding property; in other words, the biorthogonality condition on the columns of $\mat Q$ and $\mat P$ is replaced by an equivalent condition on their sketches. This randomized approach is computationally less expensive than the classical two-sided Gram-Schmidt process, has better numerical stability, and the condition number of the computed bases $\mat Q, \mat P$ is often smaller than in the deterministic case. Several different implementations of the randomized algorithm are analyzed and compared numerically. 
    The randomized two-sided Gram-Schmidt process is applied to the nonsymmetric Lancozs algorithm for the approximation of eigenvalues and both left and right eigenvectors.
\end{abstract}

\section{Introduction}
\label{sec:intro}

The Gram-Schmidt process is a fundamental tool in numerical linear algebra for the orthonormalization of a set of vectors. It is used in several applications, including in the Arnoldi process for the construction of an orthonormal basis of a Krylov subspace, which is used for a variety of problems, such as the numerical solution of linear systems or the computation of approximate eigenvalues and eigenvectors. The Gram-Schmidt process can be implemented in several mathematically equivalent ways, which have different performance in finite precision arithmetic; for example, the classical Gram-Schmidt algorithm (CGS) suffers from numerical instability, but it can be implemented using matrix-vector operations, and it is thus suitable for parallel architectures; on the other hand, modified Gram-Schmidt (MGS) is more numerically stable, but it mainly uses vector-vector operations, making it more difficult to parallelize; see e.g. \cite{GolubVanLoan13, Bjorck67, BjorkPaige92}.

In recent years, randomized techniques have been used to develop a randomized Gram-Schmidt algorithm, which uses randomized sketching to compute a basis that is sketch-orthogonal instead of orthogonal. This algorithm is cheaper than the standard Gram-Schmidt process and can be implemented in a numerically stable way \cite{BalabanovGrigori22, BalabanovGrigori25}. This approach has demonstrated its effectiveness and has been implemented within Krylov subspace methods for the solution of a range of different numerical linear algebra problems \cite{NakatsukasaTropp24, GuttelSchweitzer23, PalittaSchweitzerSimoncini25,DeDamasGrigori24,dedamas2025randomizedkrylovschureigensolverdeflation}. An unconditionally stable randomized Householder QR algorithm \cite{grigori2024randomizedhouseholderqr} can also be used in this context.

In this work, we concentrate on the two-sided Gram-Schmidt process \cite{Parlett92, FGM93}, which is a variant of the standard Gram-Schmidt process that biorthogonalizes two sets of vectors, instead of orthogonalizing a single set. Specifically, given two matrices $\mat X$, $\mat Y \in \R^{n \times m}$, the two-sided Gram-Schmidt process computes $\mat Q$, $\mat P \in \R^{n \times m}$ so that $\vspan(\mat Q) = \vspan(\mat X)$, $\vspan(\mat P) = \vspan(\mat Y)$ and $\mat P^T \mat Q = \mat I$, that is, $\mat Q$ and $\mat P$ are biorthogonal. Biorthogonal sets of vectors appear, for instance, in the nonsymmetric Lanczos algorithm \cite{Lanczos50}.

The main focus of this paper is a randomized variant of the two-sided Gram-Schmidt process, which computes two sets of vectors that are sketch-biorthogonal, with an approach similar to the randomized Gram-Schmidt process for sketch-orthogonalization \cite{BalabanovGrigori22}. Compared to the standard two-sided Gram-Schmidt process, this randomized variant is computationally less expensive and typically generates better conditioned bases. We provide a theoretical analysis to explain this effect and we support our results with illustrative numerical experiments. 
The algorithm is applied within the nonsymmetric Lanczos process for the computation of eigenvalues and eigenvectors, where it exhibits competitive performance relative to the standard algorithm.

The remainder of the paper is organized as follows. We begin by reviewing in \cref{sec:preliminaries} some basic notions that will be used throughout this work, including orthogonal and oblique projectors, the two-sided Gram-Schmidt process, and randomized subspace embeddings. We present the randomized two-sided Gram-Schmidt process in \cref{sec:randomized-two-sided-gram-schmidt}, starting from sketched oblique projectors and then describing the algorithm and its variants. In \cref{sec:theoretical-analysis} we investigate the behavior of the randomized two-sided Gram-Schmidt process and compare it against the standard two-sided Gram-Schmidt algorithm. In \cref{sec:application-to-nonsymmetric-lanczos} we apply our algorithm to develop a randomized nonsymmetric Lanczos algorithm for eigenvalue approximation, and we discuss its relation with the classical nonsymmetric Lanczos algorithm. The performance of the two-sided Gram-Schmidt process is illustrated in \cref{sec:numerical-experiments} with some numerical experiments, and concluding remarks are given in \cref{sec:conclusions}.

\subsection{Notation}
\label{subsec:notation}

We use bold uppercase letters to denote matrices, and bold lowercase letters to denote vectors.
Given a matrix $\mat A \in \R^{n \times m}$, we will denote its columns by $\vec a_1, \dots, \vec a_m \in \R^n$. We will use the notation $\mat A_i = [\vec a_1, \vec a_2, \dots, \vec a_i]$ to denote the matrix formed by the first $i$ columns of~$\mat A$, and denote the span of the first $i$ columns of $\mat A$ by $\mathcal{A}_i = \range(\mat A_i) \subset \R^n$. We denote the Euclidean inner product between two vectors $\vec x$ and $\vec y$ by $\scpr{\vec x, \vec y} = \vec x^T \vec y$. If $\vec x$ is orthogonal to a subspace $\mathcal{Q} \subset \R^n$, i.e., $\scpr{\vec x, \vec q} = 0$ for all $\vec q \in \mathcal{Q}$, we also write $\vec x \perp \mathcal{Q}$ or $\vec x \in \mathcal{Q}^\perp$. Given a matrix $\sketch \in \R^{s \times n}$, we write $\vec x \perp_\sketch \mathcal{Q}$ if $\sketch \vec x \perp \sketch \mathcal{Q}$, and similarly $\vec x \in \mathcal{Q}^{\perp_\sketch}$. We denote by $\norm{\vec x}$ the Euclidean norm of the vector $\vec x$.     

\section{Preliminaries}
\label{sec:preliminaries}

In this section, we recall basic notions regarding orthogonal and oblique projections, and we establish some related notation. We then describe the two-sided Gram-Schmidt algorithm and recall the main properties of $\varepsilon$-subspace embeddings, which will be the foundation to derive the randomized two-sided Gram-Schmidt algorithm in \cref{sec:randomized-two-sided-gram-schmidt}.
Throughout this section and in the next ones, we will denote by $\mathcal{Q}$ and $\mathcal{P}$ two subspaces of $\R^n$ of dimension $m$, and by $\mat Q$ and $\mat P \in \R^{n \times m}$ two bases associated to $\mathcal{Q}$ and $\mathcal{P}$, respectively.  

\subsection{Orthogonal projectors}
\label{subsec:orthogonal-projectors}

Given a vector $\vec x \in \R^n$, there exists a unique vector $\orthproj{Q} \vec x \in \mathcal{Q}$ such that $\vec x - \orthproj{Q} \vec x \perp \mathcal{Q}$. The vector $\orthproj{Q}\vec x$ is the orthogonal projection of $\vec x$ onto $\mathcal{Q}$, and we call the operator $\orthproj{Q} \in \R^{n \times n}$ the orthogonal projector onto $\mathcal{Q}$. The following proposition summarizes some elementary properties of the orthogonal projector, see e.g. \cite{GolubVanLoan13, BBS02}.
\begin{proposition}
	\label{prop:orthogonal-projector-properties}
Let $\mathcal{Q} \subset \R^n$ be a subspace of dimension $m$ and let $\mat Q \in \R^{n \times m}$ be a basis of~$\mathcal{Q}$. Denoting by $\orthproj{Q}$ the orthogonal projector onto $\mathcal{Q}$, the following properties hold:
	\begin{enumerate}[label=(\roman*)]
		\item 
		For any $\vec x \in \R^n$, there exists a unique $\vec z \in \mathcal{Q}$ such that $\vec x - \vec z \perp \mathcal{Q}$, so the projector $\orthproj{Q}$ is well defined.
		\item
		For any $\vec x$ and $\vec y \in \R^n$, we have $\scpr{\orthproj{Q} \vec x, \vec y} = \scpr{\vec x, \orthproj{Q} \vec y}$.
		\item For any $\vec x \in \R^n$, its orthogonal projection onto $\mathcal{Q}$ satisfies  
		\begin{equation*}
			\orthproj{Q} \vec x = \argmin_{\vec z \in \mathcal{Q}} \norm{\vec x - \vec z}.
		\end{equation*} 
		\item We have $\orthproj{Q} = \mat Q \mat Q^\dagger = \mat Q (\mat Q^T \mat Q)^{-1} \mat Q^T$. In particular, if the basis $\mat Q$ is orthonormal, we have $\orthproj{Q} = \mat Q \mat Q^T$.
	\end{enumerate}
\end{proposition} 

The orthogonal projection $\orthproj{Q} \vec x$ is the unique element of $\mathcal{Q}$ such that $\vec x = \orthproj{Q} \vec x + \vec y$, with $\vec y \perp \mathcal{Q}$, so it is naturally associated with the decomposition of $\R^n$ as the direct sum $\mathcal{Q} \oplus \mathcal{Q}^\perp$.

\subsection{Oblique projectors}
\label{subsec:oblique-projectors}

Let us assume that the subspaces $\mathcal{Q}$ and $\mathcal{P}$ are such that $\mathcal{Q} \cap \mathcal{P}^\perp = \{ 0 \}$. Then we can define the oblique projector $\proj{Q}$ as the projector onto $\mathcal{Q}$ which projects orthogonally with respect to $\mathcal{P}$, that is, for all $\vec x \in \R^n$ we define $\proj{Q} \vec x \in \mathcal{Q}$ as the unique element such that $\vec x - \proj{Q} \vec x \perp \mathcal{P}$. Similarly, we define $\proj{P}$ as the oblique projector onto $\mathcal{P}$ such that $\vec x - \proj{P} \vec x \perp \mathcal{Q}$ for all $\vec x \in \R^n$.    
We remark that the notation $\proj{Q}$, $\proj{P}$ for oblique projectors is potentially ambiguous, as it omits the subspace against which we impose the orthogonality condition. However, throughout this work we only consider paired subspaces such as $\mathcal{Q}$ and $\mathcal{P}$, so the orthogonal subspace will always be evident from the context.

The following proposition summarizes some basic properties of the oblique projectors. We include its proof for completeness.

\begin{proposition}
	\label{prop:oblique-projector-properties}
	Let $\mathcal{Q}$, $\mathcal{P} \subset \R^n$ be subspaces of dimension $m$ such that $\mathcal{Q} \cap \mathcal{P}^\perp = \{ 0 \}$, and let $\mat Q$ and $\mat P \in \R^{n \times m}$ be bases of $\mathcal{Q}$ and $\mathcal{P}$ respectively. Denoting by $\proj{Q}$ and $\proj{P}$ the oblique projectors onto $\mathcal{Q}$ and $\mathcal{P}$ as defined above, the following properties hold:
	\begin{enumerate}[label=(\roman*)]
		\item 
		For any $\vec x \in \R^n$, there exists a unique $\vec z \in \mathcal{Q}$ such that $\vec x - \vec z \perp \mathcal{P}$, so the projector $\proj{Q}$ is well defined.
		\item
		For any $\vec x$ and $\vec y \in \R^n$, we have $\scpr{\proj{Q} \vec x, \vec y} = \scpr{\vec x, \proj{P} \vec y}$.
		\item For any $\vec x \in \R^n$, if $\mat Q^T \mat P = I$ we have 
		\begin{equation*}
			\proj{Q} \vec x = \argmin_{\vec z \in \mathcal{Q}} \norm{\mat P^T (\vec x - \vec z)}.
		\end{equation*} 
		\item We have $\proj{Q} = \mat Q (\mat P^T \mat Q)^{-1} \mat P^T$. In particular, if $\mat Q^T \mat P = I$ we have $\proj{Q} = \mat Q \mat P^T$.
	\end{enumerate}
\end{proposition} 
\begin{proof}
    For $(i)$, we have $\mathcal{Q} \cap \mathcal{P}^\perp = \{ 0 \}$ and therefore $\bb{R}^n = \mathcal{Q} \oplus \mathcal{P}^\perp$ holds. Hence, any vector $\vec x\in\bb{R}^{n}$ can be written in a unique way as $\vec x = \vec z + \vec v$, where $\vec z \in\mathcal{Q}$ and $\vec v \in \mathcal{P}^\perp$. This shows that we can uniquely define $\proj{Q} \vec x = \vec z$, and that the projector $\proj{Q}$ is well defined.

    For $(ii)$, using $(i)$ we can write $\vec y = \proj{P} \vec y + \vec u$ and $\vec x = \proj{Q} \vec x + \vec v$, where $\vec u\in \mathcal{Q}^\perp$ and $\vec v\in\mathcal{P}^\perp$. Thus, we have
    \begin{align*}
        \scpr{\proj{Q}\vec x, \vec y} &= \scpr{\proj{Q} \vec x, \proj{P} \vec y} + \scpr{\proj{Q}\vec x , \vec u} = \scpr{\proj{Q}\vec x, \proj{P}\vec y},\\
        \scpr{\vec x, \proj{P}\vec y} &= \scpr{\proj{Q}\vec x, \proj{P}\vec y} + \scpr{\vec v, \proj{P} \vec y} \;= \scpr{\proj{Q}\vec x, \proj{P} \vec y},
    \end{align*}
    which proves the equality.

    For $(iii)$, we write $\vec x = \proj{Q} \vec x + \vec v$, with $\vec v\in\mathcal{P}^\perp$. It follows that for any $\vec z \in \mathcal{Q}$, we have 
    \begin{equation*}
        \| \mat P^T(\vec x - \vec z) \| = \| \mat P^T ( \proj{Q} \vec x + \vec v - \vec z) \| = \| \mat P^T \proj{Q} \vec x - \mat P^T \vec z\|,
    \end{equation*}
    that is minimized by taking $\vec z = \proj{Q} \vec x$.

    To prove $(iv)$, we first show that $\mat Q (\mat P^T \mat Q)^{-1} \mat P^T$ is a projector onto $\mathcal{Q}$. Indeed, we have $\range (\mat Q (\mat P^T \mat Q)^{-1} \mat P^T) = \mathcal{Q}$ and  
    \begin{equation*}
        (\mat Q (\mat P^T \mat Q)^{-1} \mat P^T)^2 = \mat Q (\mat P^T \mat Q)^{-1} \mat P^T \mat Q (\mat P^T \mat Q)^{-1} \mat P^T = \mat Q (\mat P^T \mat Q)^{-1} \mat P^T.
    \end{equation*}
    To conclude that $\mat Q (\mat P^T \mat Q)^{-1} \mat P^T = \proj{Q}$, we just have to show that the condition in $(i)$ holds. Given a vector $\mat P \vec z \in \mathcal{P}$ and a vector $\vec x \in \R^n$, we have
    \begin{equation*}
        (\mat P \vec z)^T (\vec x - \mat Q (\mat P^T \mat Q)^{-1} \mat P^T \vec x) = \vec z^T \mat P^T \vec x - \vec z^T \mat P^T \mat Q (\mat P^T \mat Q)^{-1} \mat P^T \vec x = 0.
    \end{equation*}
	Then $\mat Q (\mat P^T \mat Q)^{-1} \mat P^T = \proj{Q}$ because of the uniqueness of the projector which satisfies $(i)$. If $\mat Q^T \mat P=I$, it follows immediately that $\proj{Q} = \mat Q \mat P^T$. This concludes the proof.
\end{proof}

Similarly to the case of the orthogonal projector $\orthproj{Q}$, the oblique projection $\proj{Q} \vec x$ is the unique element of~$\mathcal{Q}$ such that $\vec x = \proj{Q} \vec x + \vec v$, with $\vec v \perp \mathcal{P}$, so it is naturally associated with the decomposition of $\R^n$ as the direct sum $\mathcal{Q} \oplus \mathcal{P}^\perp$.

\subsection{Two-sided Gram-Schmidt algorithm}
\label{subsec:two-sided-gram-schmidt}

In this section we review the two-sided Gram-Schmidt algorithm for biorthogonalization. Let $\mat X$ and $\mat Y \in \R^{n \times m}$ be two given matrices with $m < n$, and assume that they have full rank. Our goal is to construct two biorthogonal matrices $\mat Q$ and $\mat P \in \R^{n \times m}$, such that $\mat Q^T \mat P = I$ and $\range(\mat Q) = \range(\mat X)$, $\range(\mat P) = \range(\mat Y)$.
The algorithm is quite similar to the standard Gram-Schmidt process for orthogonalization, with the main difference being that oblique projectors are used in place of orthogonal projectors.

We start by introducing some notation that will be used throughout the paper. Let $\mathcal{Q} = \range(\mat X)$ and $\mathcal{P} = \range(\mat Y)$, and respectively denote by $\proj{Q}$ and $\proj{P}$ the oblique projectors onto $\mathcal{Q}$ and $\mathcal{P}$ that project orthogonally with respect to the other subspace. Similarly, for each $i = 1, \dots, m$, we define the subspaces spanned by the first $i$ columns of $\mat X$ and $\mat Y$ as $\mathcal{Q}_i = \range(\mat X_i)$ and $\mathcal{P}_i = \range(\mat Y_i)$, and we respectively denote by $\proj{Q}_i$ and $\proj{P}_i$ the oblique projectors onto $\mathcal{Q}_i$ and $\mathcal{P}_i$ acting orthogonally with respect to the other subspace.

In the $i$-th iteration of the two-sided Gram-Schmidt algorithm, the $i$-th column $\vec q_i$ of $\mat Q$ is computed by subtracting from $\vec x_i$ its oblique projection $\proj{Q}_{i-1} \vec x_i$, which ensures that $\vec q_i$ is orthogonal to $\mathcal{P}_{i-1}$. The $i$-th column $\vec p_i$ of $\mat P$ is computed in an analogous way, and the vectors $\vec q_i$ and $\vec p_i$ are then normalized by imposing $\scpr{\vec q_i, \vec p_i} = 1$. The algorithm is outlined in \cref{algorithm:two-sided-gram-schmidt}. In practice, the oblique projector $\proj{Q}_{i-1}$ can be implemented in several different mathematically equivalent ways, and this choice can have a significant impact on the numerical behavior of the algorithm. We present the most common implementations for the projector in the following sections.

\begin{remark}
	\label{rem:two-sided-gram-schmidt-breakdown}
	If $\scpr{\vec q_i, \vec p_i} = 0$ at line $4$ of \cref{algorithm:two-sided-gram-schmidt}, the algorithm breaks down and it is not possible to continue the biorthogonalization process. In the following, we are going to assume that this breakdown never occurs. However, even small nonzero values of $\scpr{\vec q_i, \vec p_i}$ can be problematic, because after normalization the $i$-th columns of $\mat Q$ and $\mat P$ would have very large norms; this is one of the main reasons for the quick growth in the condition number of the bases $\mat Q$ and $\mat P$.
\end{remark}

\begin{remark}
	\label{rem:two-sided-gram-schmidt-normalizations}
	In our implementation of \cref{algorithm:two-sided-gram-schmidt}, we normalize the vectors $\vec q_i$ and $\vec p_i$ in order to have $\scpr{\vec q_i, \vec p_i} = 1$ and $\norm{\vec q_i} = \norm{\vec p_i}$. Other choices for normalization are possible, for example one could normalize $\norm{\vec q_i} = \norm{\vec p_i} = 1$, and obtain matrices $\mat Q$ and $\mat P$ which satisfy the biorthogonality condition $\mat Q^T \mat P = \diag(d_1, \dots, d_m)$, with $d_i = \scpr{\vec q_i, \vec p_i}$. Since we did not notice any significant differences in the finite precision behavior of the algorithm when employing different normalization strategies, we opted to use the normalization illustrated in \cref{algorithm:two-sided-gram-schmidt} in order to simplify the presentation.
\end{remark}

\begin{algorithm}[t]
	\caption{Two-sided Gram-Schmidt process
	\label{algorithm:two-sided-gram-schmidt}}
	\begin{algorithmic}[1]
	\REQUIRE $\mat X$, $\mat Y \in \R^{n\times m}$.
	\ENSURE $\mat Q$, $\mat P \in \R^{n\times m}$ such that $\range(\mat Q) = \range(\mat X)$, $\range(\mat P) = \range(\mat Y)$ and $\mat Q^T \mat P = I$.   
	\FOR {$i=1, \ldots, m$}
	\STATE $\vec q_i = \vec x_i - \proj{Q}_{i-1} \vec x_i$   
	\STATE $\vec p_i = \vec y_i - \proj{P}_{i-1} \vec y_i$   
	\STATE $d_i = \scpr{\vec q_i, \vec p_i}$ 
	\IF{$d_i = 0$}
		\RETURN	\COMMENT{breakdown}
	\ENDIF 
	\STATE $\vec q_i = \vec q_i / \sqrt{d_i}$
	\STATE $\vec p_i = \vec p_i / \sqrt{d_i} \cdot \sign(d_i)$ 
	\ENDFOR
	\STATE $\mat Q = [\vec q_1, \dots, \vec q_m]$ 
	\STATE $\mat P = [\vec p_1, \dots, \vec p_m]$ 
	\end{algorithmic}
\end{algorithm}
	
\subsubsection{Classical Gram-Schmidt}
\label{subsubsec:classical-gram-schmidt}

Observe that \cref{algorithm:two-sided-gram-schmidt} constructs biorthogonal matrices $\mat Q$ and $\mat P$ such that $\mat Q_i^T \mat P_i = I \in \R^{i \times i}$ for all $i = 1, \dots, m$. Therefore, using \cref{prop:oblique-projector-properties}$(iv)$, the oblique projector onto $\mathcal{Q}_i$ can be computes as $\proj{Q}_i = \mat Q_i \mat P_i^T = \sum_{j = 1}^i \vec q_j \vec p_j^T$, and similarly for $\proj{P}_i$. 
With this approach, lines 2 and 3 of \cref{algorithm:two-sided-gram-schmidt} become
\begin{equation*}
	\begin{aligned}
		\vec q_i &= \vec x_i - \mat Q_{i-1} \mat P_{i-1}^T \vec x_i, \\
		\vec p_i &= \vec y_i - \mat P_{i-1} \mat Q_{i-1}^T \vec y_i.
	\end{aligned}
\end{equation*}
The resulting algorithm is known as the classical Gram-Schmidt process (CGS), and its main advantage is the fact that all inner products $\mat P_{i-1}^T \vec x_i$ and $\mat Q_{i-1}^T \vec y_i$ can be computed simultaneously, making it suitable for parallel architectures. However, this procedure can suffer from numerical instabilities, and it may be necessary to repeat the orthogonalization to ensure that biorthogonality holds in finite precision arithmetic.

\subsubsection{Modified Gram-Schmidt}
\label{subsubsec:modified-gram-schmidt}

The modified Gram-Schmidt process (MGS) is an alternative to CGS which is more numerically stable, but it involves computing the inner products with the columns of $\mat Q_i$ and $\mat P_i$ sequentially, so it cannot be parallelized as easily. Its formulation can be derived from the following simple observation: since $\vec p_\ell^T \vec q_j = 0$ for all $\ell \ne j$, we can write 
\begin{equation*}
	I - \proj{Q}_i = \prod_{j = 1}^i (I - \vec q_j \vec p_j^T) \qquad \text{and} \qquad I - \proj{P}_i = \prod_{j = 1}^i ( I - \vec p_j \vec q_j^T ).
\end{equation*}  
Notice that $\vec q_j \vec p_j^T$ is the oblique projector onto $\vspan \{\vec q_j\}$ that acts orthogonally with respect to $\vspan \{ \vec p_j \}$, so using this implementation corresponds to sequentially subtracting from $\vec x_i$ its oblique projection onto each of the columns of $\mat Q_{i-1}$.
The resulting algorithm has better numerical stability with respect to CGS.

\subsubsection{Computational cost}
\label{subsubsec:computational-cost-ts-gram-schmidt}

The computational costs of CGS and MGS are essentially the same, and they are roughly double the cost of the standard Gram-Schmidt process, due to the involvement of two vectors per iteration. In each iteration, the main costs are the computation of the inner products of $\vec x_i$ (or its update, in the case of MGS) with the basis $\mat P_{i-1}$, and the computation of $\vec q_i$ as a linear combination of $\vec x_i$ and the columns of $\mat Q_{i-1}$, as well as the corresponding operations for the computation of $\vec y_i$. So the $i$-th iteration costs $\mathcal{O}(ni)$, and the overall cost for $m$ iterations is $\mathcal{O}(nm^2)$ for both CGS and MGS. The main advantage of CGS is its reliance of level-2 BLAS operations (matrix-vector products), which can be easily parallelized; on the other hand, MGS uses sequential level-1 BLAS operations (vector-vector products), which are more difficult to parallelize and slower on most computational architectures. The trade-off is that MGS has better numerical stability and, for this reason, it is often preferred to CGS. For both implementations, re-biorthogonalization can be employed to improve the stability of the two methods, by repeating lines 2 and 3 of \cref{algorithm:two-sided-gram-schmidt} multiple times to improve the biorthogonality of $\mat Q$ and $\mat P$ in finite precision. We refer to the variants of \cref{algorithm:two-sided-gram-schmidt} in which biorthogonalization using CGS (MGS) is repeated two or three times as CGS2 and CGS3 (MGS2 and MGS3), respectively.

Both CGS and MGS heavily rely on the fact that $\mat P_{i-1}$ and $\mat Q_{i-1}$ are biorthogonal in order to simplify the computation of $\proj{Q}_{i-1}$ and $\proj{P}_{i-1}$ in each iteration of \cref{algorithm:two-sided-gram-schmidt}. However, due to numerical errors this biorthogonality property does not hold exactly in finite precision, and hence $\mat P_{i-1}^T \mat Q_{i-1}$ can be quite far from the identity matrix in practice. An alternative way to account for this error would be to explicitly compute the projector $\proj{Q}_{i-1} = \mat Q_{i-1} (\mat P_{i-1}^T \mat Q_{i-1})^{-1} \mat P_{i-1}^T$. However, this approach would require updating the matrix $\mat P_{i-1}^T \mat Q_{i-1}$ at each iteration, which involves computing about $2i$ inner products in the $i$-th iteration, thus doubling the number of inner products with respect to the CGS algorithm. Instead, these inner products could have been used to update the previous columns $\vec q_{i-1}$ and $\vec p_{i-1}$ with a CGS2 implementation, so it is debatable whether it would be beneficial to explicitly apply the projectors without assuming that $\mat Q_{i-1}$ and $\mat P_{i-1}$ are biorthogonal.
We still mention this possibility now because it is going to be more viable in the randomized algorithm, where inner products are replaced by significantly cheaper sketched inner products.

\subsection{Subspace embeddings}
\label{subsec:subspace-embeddings}

In this section, we recall some basic properties of subspace embeddings, which will be useful in deriving the randomized two-sided Gram-Schmidt process in \cref{sec:randomized-two-sided-gram-schmidt}.

We recall that a sketching matrix $\sketch \in \R^{s \times n}$ is said to be an $\varepsilon$-subspace embedding for a subspace $\mathcal{Q} \subset \R^n$  with $\varepsilon \in (0, 1)$ if we have
\begin{equation}
	\label{eqn:epsilon-embedding-property}
	\abs{\scpr{\vec x, \vec y} - \scpr{\sketch \vec x, \sketch \vec y}} \le \varepsilon \norm{\vec x} \norm{\vec y} \qquad \text{for all }\vec x, \vec y \in \mathcal{Q}.
\end{equation}
Since $\sketch$ approximately preserves the inner products of vectors in $\mathcal{Q}$, it also approximately preserves the extremal singular values of a matrix $\mat Q$ whose columns span the subspace $\mathcal{Q}$. Indeed, we have \cite[Lemma~70]{Woodruff14}
\begin{equation}
\label{eqn:epsilon-embedding-singvals}
\begin{aligned}
	(1 - \varepsilon)^{1/2} \sigma_{\max}(\mat Q) &\le \sigma_{\max}(\sketch \mat Q) \le (1 + \varepsilon)^{1/2} \sigma_{\max}(\mat Q), \\
	(1 - \varepsilon)^{1/2} \sigma_{\min}(\mat Q) &\le \sigma_{\min}(\sketch \mat Q) \le (1 + \varepsilon)^{1/2} \sigma_{\min}(\mat Q),
\end{aligned}    
\end{equation}
and therefore
\begin{equation}
	\label{eqn:epsilon-embedding-condition-number}
	\left(\frac{1-\varepsilon}{1+\varepsilon}\right)^{1/2} \kappa(\sketch \mat Q) \le \kappa(\mat Q) \le \left(\frac{1+\varepsilon}{1-\varepsilon}\right)^{1/2} \kappa(\sketch \mat Q).
\end{equation}
A bound similar to \cref{eqn:epsilon-embedding-singvals} also holds for other non-extremal singular values, see for instance \cite[Theorem~2.2]{GrigoriXue25}.

In practical settings, we often want a sketching matrix $\sketch$ to satisfy \cref{eqn:epsilon-embedding-property} for an unknown subspace $\mathcal{Q}$, so $\sketch$ is usually drawn at random according to a certain distribution, in order to satisfy $\cref{eqn:epsilon-embedding-property}$ with high probability for any subspace of a certain dimension $m$; such a sketching matrix is known in the literature as an oblivious $\varepsilon$-subspace embedding. It is also desirable for $\sketch$ to be cheap to apply. In this work, we use as sketching matrix $\sketch$ a sparse sign matrix, but several alternatives have been considered in the literature, such as for example randomized subsampled Fourier, cosine or Hadamard transforms (see, e.g., \cite{BalabanovGrigori22} and \cite{NakatsukasaTropp24}).

\subsubsection{Sparse sign matrix}
\label{subsubsec:sparsematrix}

The sparse sign matrix \cite{martinsson_tropp_2020, meng_mahoney_2013} is a common choice due to its sparsity, which makes its action on a vector very efficient to compute. Given a fixed sparsity parameter $\zeta$ with $2\le \zeta \le s$, the matrix $\sketch \in\mathbb{R}^{s\times n}$ is constructed as
\begin{equation*}
    \sketch = \sqrt{\frac{n}{\zeta}} [ \vec s_1, \ldots, \vec s_n ],
\end{equation*}
where the columns $\vec s_i \in \mathbb{R}^s$ are i.i.d.~random vectors. To build each of the vectors, starting from $\vec s_i = \vec 0$, we choose $\zeta$ random entries and assign to them a random value in $\{-1, 1\}$, while leaving the other entries equal to $0$. In \cite{TYUC2019}, a recommended choice is $\zeta = \min\{s,8\}$. The sparse sign matrix~$\sketch$ can be applied to a vector in $\mathbb{R}^n$ with $\mathcal{O}(\zeta n)$ arithmetic operations.
It was shown in \cite{Cohen16} that a sparse sign matrix is an oblivious $\varepsilon$-subspace embedding for subspaces of dimension $m$ when $s = \mathcal{O}(\varepsilon^{-2} m \log m)$ and $\xi = \mathcal{O}(\varepsilon^{-1} \log m)$.

\section{Randomized two-sided Gram-Schmidt process}
\label{sec:randomized-two-sided-gram-schmidt}

In this section we present the randomized two-sided Gram-Schmidt algorithm, which employs randomized sketching to compute bases $\mat P$ and $\mat Q$ that satisfy the sketched biorthogonality condition $(\sketch \mat Q)^T \sketch \mat P = I$, where $\sketch \in \R^{s \times n}$ is a sketching matrix that is an $\varepsilon$-subspace embedding for $\range(\mat X)$ and $\range(\mat Y)$. We begin by introducing sketched oblique projectors and their properties in order to lay the theoretical foundation for the algorithm.

\subsection{Sketched oblique projectors}
\label{subsec:sketched-oblique-projectors}

Let us assume that the two subspaces $\mathcal{Q}$ and $\mathcal{P}$ are such that $\sketch \mathcal{Q} \cap (\sketch \mathcal{P})^\perp = \{ 0 \}$. 
We can define the sketched oblique projector $\skproj{Q}$ as the projector onto $\mathcal{Q}$ that projects sketch-orthogonally with respect to $\mathcal{P}$, that is, for any $\vec x \in \R^n$ we define $\skproj{Q}\vec x \in \mathcal{Q}$ as the unique element of $\mathcal{Q}$ such that $\vec x - \skproj{Q} \vec x \perp_\sketch \mathcal{P}$. 
Similarly, we define $\skproj{P}$ as the oblique projector onto $\mathcal{P}$ such that $\vec x - \skproj{P}\vec x \perp_\sketch \mathcal{Q}$ for all $\vec x \in \R^n$.  

The following proposition summarizes the main properties of the sketched oblique projector, which closely resemble the ones given in \cref{prop:oblique-projector-properties} for the oblique projector.

\begin{proposition}
	\label{prop:sketched-oblique-projector-properties}
	Let $\mathcal{Q}$, $\mathcal{P} \subset \R^n$ be subspaces of dimension $m$ and let $\mat Q$, $\mat P \in \R^{n \times m}$ be bases of $\mathcal{Q}$ and $\mathcal{P}$, respectively. Let $\sketch \in \R^{s \times n}$ be an $\varepsilon$-subspace embedding for $\mathcal{Q}$ and $\mathcal{P}$ for some $\varepsilon \in (0,1)$, such that $\sketch \mathcal{Q} \cap (\sketch \mathcal{P})^\perp = \{ 0 \}$. Denoting by $\skproj{Q}$ and $\skproj{P}$ the sketched oblique projectors onto $\mathcal{Q}$ and $\mathcal{P}$ as defined above, the following properties hold:
	\begin{enumerate}[label=(\roman*)]
		\item 
		For any $\vec x \in \R^n$, there exists a unique $\vec z \in \mathcal{Q}$ such that $\vec x - \vec z \perp_\sketch \mathcal{P}$, so the projector $\skproj{Q}$ is well defined.
		\item
		For any $\vec x$ and $\vec y \in \R^n$, we have $\scpr{\sketch \skproj{Q} \vec x, \sketch \vec y} = \scpr{\sketch \vec x, \sketch \skproj{P} \vec y}$.
		\item For any $\vec x \in \R^n$, if $(\sketch \mat Q)^T \sketch \mat P = I$ we have 
		\begin{equation*}
			\skproj{Q} \vec x = \argmin_{\vec z \in \mathcal{Q}} \norm{(\sketch \mat P)^T \sketch (\vec x - \vec z)}.
		\end{equation*} 
		\item We have $\skproj{Q} = \mat Q ((\sketch \mat P)^T \sketch \mat Q)^{-1} (\sketch \mat P)^T \sketch$. In particular, if $(\sketch \mat Q)^T \sketch \mat P = I$ we have $\skproj{Q} = \mat Q (\sketch \mat P)^T \sketch$.
	\end{enumerate}
\end{proposition} 
\begin{proof}
    To prove $(i)$, observe that since $\sketch \mathcal{Q} \cap (\sketch \mathcal{P})^\perp = \{ 0 \}$, for any vector $\vec x\in\bb{R}^n$ there is a unique representation $\sketch \vec x = \vec u + \vec v$, where $\vec u \in\sketch \mathcal{Q}$ and $\vec v \in (\sketch \mathcal{P})^\perp$. We can write $\vec u = \sketch \vec z$ for some $\vec z \in \mathcal{Q}$, so $\vec x - \vec z \perp_\sketch \mathcal{P}$. To conclude that $\skproj{Q} \vec x = \vec z$ is well defined, it only remains to show that $\vec z$ is unique. If we assume that we also have $\vec u = \sketch \vec w$, with $\vec w \ne \vec z$, $\vec w \in \mathcal{Q}$, we would obtain $\sketch (\vec z - \vec w) = \vec 0$, which contradicts the $\varepsilon$-subspace embedding property for $\sketch$. The uniqueness of $\vec z$ then follows from the uniqueness of $\vec u$.  
	
	For $(ii)$, we know from $(i)$ that we can write $\sketch \vec y = \sketch \skproj{P}\vec y + \vec v$, with $\vec v\in(\sketch \mathcal{Q})^\perp$, and $\sketch \vec x = \sketch \skproj{Q} \vec x + \vec u$, with $\vec u\in(\sketch \mathcal{P})^\perp$. This leads to
    \begin{align*}
        \scpr{\sketch \skproj{Q}\vec x, \sketch \vec y} &= \scpr{\sketch \skproj{Q}\vec x, \sketch \skproj{P}\vec y} + \scpr{\sketch \skproj{Q}\vec x, \vec v} = \scpr{\sketch \skproj{Q}\vec x, \sketch \skproj{P}\vec y},\\
        \scpr{\sketch \vec x, \sketch \skproj{P}\vec y} &= \scpr{\sketch\skproj{Q}\vec x, \sketch \skproj{P}\vec y} + \scpr{\vec u, \sketch \skproj{P}\vec y} = \scpr{\sketch \skproj{Q}\vec x, \sketch \skproj{P}\vec y},
    \end{align*}
    proving the equality.

    For $(iii)$, we can write $\vec x = \skproj{Q}\vec x + \vec u$ such that $\sketch \vec u\in (\sketch \mathcal{P})^\perp$. It follows that
    \begin{equation*}
        \| (\sketch \mat P)^T \sketch (\vec x - \vec z) \| = \| (\sketch \mat P)^T (\sketch \skproj{Q}\vec x + \sketch \vec u - \sketch \vec z) \| = \| (\sketch \mat P)^T \sketch \skproj{Q} \vec x - (\sketch \mat P)^T \sketch \vec z \|,
    \end{equation*} which is minimized by taking $\vec z = \skproj{Q}\vec x$.

    For $(iv)$, first we check that $\mat Q ((\sketch \mat Q)^T \sketch \mat P)^{-1} (\sketch \mat P)^T \sketch$ is a projector onto $\mathcal Q$. Indeed, we have $\range(\mat Q ((\sketch \mat Q)^T \sketch \mat P)^{-1} (\sketch \mat P)^T \sketch) = \mathcal{Q}$ and 
    \begin{equation*}
        \left(\mat Q ((\sketch \mat Q)^T \sketch \mat P)^{-1} (\sketch \mat P)^T \sketch \right) \mat Q = \mat Q.
    \end{equation*}
    So we just have to show that the condition in $(i)$ holds. Given a vector $\vec x \in\bb{R}^n$, and any vector $\vec y\in\sketch \mathcal{P}$, we can write $\vec y = \sketch \mat P \vec z$ and we have
    \begin{equation*}
        (\sketch \mat P \vec z)^T \sketch \vec x - (\sketch \mat P \vec z)^T \sketch \mat Q ((\sketch \mat P)^T \sketch \mat Q)^{-1} (\sketch \mat P)^T \sketch \vec x = \vec z^T(\sketch \mat P)^T \sketch \vec x - \vec z^T(\sketch \mat P)^T \sketch \vec x = 0.
    \end{equation*}
    This proves that $\mat Q ((\sketch \mat Q)^T \sketch \mat P)^{-1} (\sketch \mat P)^T \sketch = \skproj{Q}$. 
    If $(\sketch \mat Q)^T \sketch \mat P = I$, it follows immediately that $\skproj{Q} = \mat Q(\sketch \mat P)^T \sketch$. This concludes the proof.
\end{proof}

Note that the assumption $\sketch \mathcal{Q} \cap (\sketch \mathcal{P})^\perp = \{ 0 \}$ is equivalent to $\mathbb{R}^{s} = \sketch\mathcal{Q} \oplus \sketch\mathcal{P}^\perp$, which in turn implies $\mathbb{R}^n = \mathcal{Q} \oplus \mathcal{P}^{\perp_\sketch}$, where $\mathcal{P}^{\perp_\sketch} := \{ \vec v \in \R^n : \scpr{\sketch \vec v, \sketch\vec p} = 0 \: \forall \vec p \in \mathcal{P}\}$. To prove this, assume we have a vector $\vec u \in\mathcal{Q}\cap \mathcal{P}^{\perp_\sketch}$, $\vec u \ne 0$. Then by definition $\sketch \vec u\in\sketch\mathcal{Q}$, and since $\vec u \in\mathcal{P}^{\perp_\sketch}$, for any $\vec p\in \mathcal{P}$ we have $\scpr{\sketch \vec u, \sketch \vec p} = 0$, so $\sketch \vec u \in (\sketch \mathcal{P})^\perp$. Since $\sketch \mathcal{Q} \cap (\sketch \mathcal{P})^\perp = \{ 0 \}$, we conclude that $\sketch \vec u = \vec 0$.
However, the $\varepsilon$-subspace embedding property implies that
\begin{equation*}
	\norm{\vec u}^2 = \abs{ \norm{\vec u}^2 - \norm{\sketch \vec u}^2 } \le \varepsilon \norm{\vec u}^2,
\end{equation*}
which is a contradiction for any $\varepsilon \in (0,1)$.

\subsection{General algorithm}
\label{subsec:rgs-general-algorithm}

To describe the randomized two-sided Gram-Schmidt process in its general formulation, we introduce a notation similar to the one used in \cref{subsec:two-sided-gram-schmidt}. Given the two matrices $\mat X$ and $\mat Y \in \R^{n \times m}$, recall that we denote by $\mathcal{Q}_i = \range(\mat X_i)$ and $\mathcal{P}_i = \range(\mat Y_i)$. For $i = 1, \dots, m$, we denote by $\skproj{Q}_{i}$ the sketched oblique projector onto $\mathcal{Q}_i$ that acts sketch-orthgonally with respect to $\mathcal{P}_i$, i.e., such that for any $\vec x \in \R^n$ we have $\skproj{Q}_i \vec x \perp_\sketch \mathcal{P}_i$. Likewise, we denote by $\skproj{P}_i$ the sketched oblique projector onto $\mathcal{P}_i$ that acts sketch-orthogonally with respect to $\mathcal{Q}_i$. Using this notation we can formulate a randomized two-sided Gram-Schmidt process, which constructs two sketch-biorthogonal bases $\mat Q$ and $\mat P$ for $\mathcal{Q}$ and $\mathcal{P}$, i.e.,~such that $(\sketch \mat Q)^T (\sketch \mat P) = I$. The algorithm is given in \cref{algorithm:two-sided-gram-schmidt-sketched} in its general form.
Similarly to \cref{algorithm:two-sided-gram-schmidt}, several different implementations are possible depending on how the sketched projectors $\skproj{Q}_{i-1}$ and $\skproj{P}_{i-1}$ are applied on lines 2 and 4. 

\begin{remark}
	Similarly to what happens in \cref{algorithm:two-sided-gram-schmidt}, if $\scpr{\sketch \vec q_i, \sketch \vec p_i} = 0$ then \cref{algorithm:two-sided-gram-schmidt-sketched} breaks down at iteration $i$. Again, in the following we assume that this breakdown never occurs. In \cref{sec:theoretical-analysis} we are going to analyze the behavior of $\scpr{\sketch \vec q_i, \sketch \vec p_i}$, showing that it is unlikely to be very close to $0$. This implies that the condition numbers of the bases $\mat Q_i$ and $\mat P_i$ are not expected to grow rapidly with $i$, suggesting that the sketch-biorthogonal bases will typically be better conditioned than the ones obtained from \cref{algorithm:two-sided-gram-schmidt}.
\end{remark}

\begin{algorithm}[t]
	\caption{Randomized two-sided Gram-Schmidt process
	\label{algorithm:two-sided-gram-schmidt-sketched}}
	\begin{algorithmic}[1]
	\REQUIRE $\mat X$, $\mat Y \in \R^{n\times m}$, sketching matrix $\sketch \in \R^{s \times n}$. 
	\ENSURE $\mat Q$, $\mat P \in \R^{n\times m}$ s.t.~$\range(\mat Q) = \range(\mat X)$, $\range(\mat P) = \range(\mat Y)$ and $(\sketch \mat Q)^T (\sketch \mat P) = I$.
	\FOR {$i=1, \ldots, m$}
	\STATE $\vec q_i = \vec x_i - \skproj{Q}_{i-1} \vec x_i$   
	\STATE $\vec p_i = \vec y_i - \skproj{P}_{i-1} \vec y_i$   
	\STATE $d_i = \scpr{\sketch \vec p_i, \sketch \vec q_i}$ 
	\IF {$d_i = 0$}
		\RETURN \COMMENT{breakdown}
	\ENDIF
	\STATE $\vec q_i = \vec q_i / \sqrt{d_i}$
	\STATE $\vec p_i = \vec p_i / \sqrt{d_i} \cdot \sign(d_i)$
	\ENDFOR
	\STATE $\mat Q = [\vec q_1, \dots, \vec q_m]$ 
	\STATE $\mat P = [\vec p_1, \dots, \vec p_m]$ 
	\end{algorithmic}
\end{algorithm}

\subsubsection{Randomized classical Gram-Schmidt}
\label{subsubsec:randomized-classical-gram-schmidt}

If we exploit the fact that in exact arithmetic $(\sketch \mat P_{i-1})^T \sketch \mat Q_{i-1} = I \in \R^{(i-1) \times (i-1)}$ for all $i = 1, \dots, m$, we can apply the sketched projectors as 
\begin{equation*}
	\begin{aligned}
		\skproj{Q}_{i-1} \vec x_i &= \mat Q_{i-1} (\sketch \mat P_{i-1})^T \sketch \vec x_i, \\
		\skproj{P}_{i-1} \vec y_i &= \mat P_{i-1} (\sketch \mat Q_{i-1})^T \sketch \vec y_i.
	\end{aligned}
\end{equation*} 
Due to its similarity to CGS, we refer to this variant as randomized classical Gram-Schmidt (rCGS).

The main difference bewteen the CGS and rCGS algorithms is that the latter replaces inner products with sketched inner products, so the cost for the update of $\vec q_i$ in the $i$-th iteration is reduced from $4 n i$ to $2 n i + 2 s i + \xi n$ flops, where $2si$ is the cost of the sketched inner products and $\xi n$ is the cost of applying the sparse sign sketching matrix $\sketch$, as seen in \cref{subsubsec:sparsematrix}. Since $\vec p_i$ is updated similarly to $\vec q_i$, the overall cost of an iteration of rCGS is therefore slightly more than half the cost of the corresponding CGS iteration. This means that, for example, with approximately the same cost of CGS we can implement rCGS2, i.e., apply the rCGS procedure twice, thus improving the numerical stability of the method.

\subsubsection{Randomized modified Gram-Schmidt}
\label{subsubsec:randomized-modified-gram-schmidt}

Alternatively, we can also write $I - \skproj{Q}_{i-1}$ in a similar way to MGS, by splitting the projector as a product of projectors onto the span of each single basis vector. We have
\begin{equation*}
	\skproj{Q}_{i-1} = \sum_{j = 1}^{i-1} \vec q_j (\sketch \vec p_j)^T \sketch,
\end{equation*}
and thus
\begin{equation*}
	I - \skproj{Q}_{i-1} = \prod_{j = 1}^{i-1} (I - \vec q_j (\sketch \vec p_j)^T \sketch),
\end{equation*} 
where we used the fact that $(\sketch \vec p_j)^T \sketch \vec q_\ell = 0$ for $j \ne \ell$. We refer to this implementation as randomized modified Gram-Schmidt (rMGS).
The computational cost of rMGS is the same as for rCGS, so it is roughly half the cost of MGS. 
As in the case of \cref{algorithm:two-sided-gram-schmidt}, rMGS relies on a sequence of vector-vector operations, while rCGS can be implemented using matrix-vector products. Thus, rCGS is easier to parallelize than rMGS and it is faster on most computational architectures, but this comes at a trade-off in the numerical stability of the algorithm.

\subsubsection{Randomized Gram-Schmidt with explicit oblique projection}
\label{subsubsec:randomized-gram-schmidt-with-explicit-projection}

The sketched oblique projector can be alternatively applied by computing it explicitly, without assuming that the matrix $(\sketch \mat P)^T \sketch \mat Q$ is the identity. In other words, we explicitly compute
\begin{eqnarray*}
	\skproj{Q}_{i-1} \vec x_i &=& \mat Q_{i-1} ((\sketch \mat P_{i-1})^T \sketch \mat Q_{i-1})^{-1} (\sketch \mat P_{i-1})^T \sketch \vec x_i, \\
	\skproj{P}_{i-1} \vec x_i &=& \mat P_{i-1} ((\sketch \mat Q_{i-1})^T \sketch \mat P_{i-1})^{-1} (\sketch \mat Q_{i-1})^T \sketch \vec x_i. 
\end{eqnarray*}
From here on, we are going to refer to this approach as CGS\_O for the deterministic version and rCGS\_O for the randomized one.
This approach is more expensive compared to rCGS and rMGS, as already mentioned in \cref{subsubsec:computational-cost-ts-gram-schmidt} in the deterministic case; however, the use of sketching makes the computation of the product $(\sketch \mat P_{i-1})^T \sketch \mat Q_{i-1}$ significantly less expensive than the computation of $\mat P_{i-1}^T \mat Q_{i-1}$, making this approach more competitive in the randomized setting. Specifically, the $i$-th iteration of the rCGS\_O implementation requires the computation of about $2i$ sketched inner products to update the matrix $(\sketch \mat P_{i-1})^T \sketch \mat Q_{i-1}$ in addition to the $2i$ sketched inner products computed by rCGS and rMGS, for a total of about $4i$ sketched inner products, and the solution of an $(i-1) \times (i-1)$ linear system, which costs $O(i^3)$. 
With a naive implementation, executing $m$ iterations would result in an overall computational cost of $O(m^4)$ for the linear system solves. However, by iteratively updating an LU factorization of $(\sketch \mat P_{i-1})^T \sketch \mat Q_{i-1}$, the total cost of the solves in iterations $1$ through $m$ is reduced to $O(m^3)$.  
Observe that, in contrast with the deterministic CGS\_O approach, the randomized variant has the same leading cost as CGS, since all the additional operations involve sketched vectors. Hence, especially when $n \gg m$, we expect that rCGS\_O will be only slightly slower than rCGS, but potentially more numerically stable. These observations will be confirmed by the numerical experiments in \cref{subsec:experiment-condition-number-growth}.

\section{Theoretical analysis}
\label{sec:theoretical-analysis}

In this section we obtain some theoretical results that allow us to better understand the behavior of the randomized two-sided Gram-Schmidt process.

Given two bases $\mat X, \mat Y \in \R^{n \times m}$, in this section we denote by $\mat Q$ and $\mat P$ the biorthogonal bases obtained by applying the deterministic two sided Gram-Schmidt process to $\mat X$ and $\mat Y$, and by $\mat Q^\sketch$ and $\mat P^\sketch$ the ones obtained by applying randomized two-sided Gram-Schmidt, satisfying $(\sketch \mat Q^\sketch)^T \sketch \mat P^\sketch = I$.
First of all, it immediately follows from \cref{eqn:epsilon-embedding-condition-number} that we can monitor the condition number of $\mat Q^\sketch$ and $\mat P^\sketch$ by monitoring the condition number of $\sketch \mat Q^\sketch$ and $\sketch \mat P^\sketch$, respectively. 
However, this does not imply that $\kappa(\mat Q)$ and $\kappa(\mat Q^\sketch)$ are close to each other, and it turns out that in certain cases $\mat Q^\sketch$ can be significantly better conditioned than $\kappa(\mat Q)$; in particular, this behavior is different from the standard Gram-Schmidt process for orthogonalization, where both $\mat Q$ and $\sketch \mat Q^\sketch$ have condition number equal to one. 
The discussion that follows aims to get some insight on the differences between Gram-Schmidt and randomized Gram-Schmidt, and understand in which situations one can be expected to perform better. In particular, the following proposition shows that the sketches of two (possibly orthogonal) vectors have a very small probability of being almost orthogonal. 

\begin{proposition}
	\label{prop:sketched-inner-product--gaussian}
	Let $\vec x, \vec y \in \R^n$ be vectors of unit norm, and let $\sketch \in \R^{s \times n}$ be a Gaussian sketching matrix. Then for $s \ge 2$ and any $\delta > 0$ we have 
	\begin{alignat*}{2}
		\prob (\abs{\scpr{\sketch \vec x, \sketch \vec y}} &\le \delta) \le 
		202 \, s \delta \qquad &\text{if} \quad \vec x \ne \vec y,\\
		\prob (\abs{\scpr{\sketch \vec x, \sketch \vec y}} &\le \delta) \le 
		(\delta e^{1 - \delta})^{s/2} \qquad &\text{if} \quad \vec x = \vec y.
	\end{alignat*}
\end{proposition}

\begin{proof}
	We can write $\sketch = \frac{1}{\sqrt{s}} \begin{bmatrix}
		\vec \omega_1 & \vec \omega_2 & \dots & \vec \omega_s
	\end{bmatrix}^T$, where $\vec \omega_j \in \R^n$ are independent standard Gaussian random vectors. We have
	\begin{equation*}
		Z := \scpr{\sketch \vec x, \sketch \vec y} = \frac{1}{s} \begin{bmatrix}
			\scpr{\vec \omega_1, \vec x} & \dots & \scpr{\vec \omega_s, \vec x}
		\end{bmatrix}^T \begin{bmatrix}
			\scpr{\vec \omega_1, \vec y} & \dots & \scpr{\vec \omega_s, \vec y}
		\end{bmatrix} = \frac{1}{s} \sum_{j = 1}^s \scpr{\vec \omega_j, \vec x} \scpr{\vec \omega_j, \vec y}.
	\end{equation*}
Defining $Z_j := \frac{1}{s}\scpr{\vec \omega_j, \vec x} \scpr{\vec \omega_j, \vec y}$, we have $Z = \sum_{j = 1}^s Z_j$. Note that $Z_j$ are i.i.d.~random variables, since the rows $\vec \omega_j$ of $\sketch$ are all i.i.d.~random variables.
Let us denote by $f_Z$ and $f_{Z_1}$ the probability density functions of $Z$ and $Z_1$, respectively. We have the straightforward bound
\begin{equation}
	\label{eqn:proof-sketch-ip--prob-infnorm-bound}
	\prob (\abs{\scpr{\sketch \vec x, \sketch \vec y}} \le \delta) = \int_{-\delta}^{\delta} f_Z(t) \, dt \le 2 \delta \norm{f_Z}_{\infty},
\end{equation}    
so it is sufficient to find an upper bound for $\norm{f_Z}_\infty$.  
Since the random variables $Z_j$, for $1 \le j \le s$, are independent and have the same density $f_{Z_1}$, we  have $f_Z = f_{Z_1} * \cdots * f_{Z_s} = f^{*s}_{Z_1}$. Using Young's convolution inequality, for any $k = 2, \dots, s-1$ we have
\begin{equation}
	\label{eqn:proof-sketch-ip--convolution-inequality}
	\norm{f^{*s}_{Z_1}}_\infty \le \norm{f^{*k}_{Z_1}}_{\infty} \, \norm{f^{*(s-k)}_{Z_1}}_{1} \le \norm{f_{Z_1}}_{p}^k \, \norm{f^{*(s-k)}_{Z_1}}_{1} \le \norm{f_{Z_1}}_p^k, \qquad \text{where $p = \frac{k}{k-1}$}.
\end{equation}
We are going to consider $k \ge 3$, so that we have $p \in (1,2)$. This assumption on $p$ is required in order to obtain a bound uniform in $\scpr{\vec x, \vec y}$. 

Our goal is then to obtain an upper bound for $\norm{f_{Z_1}}_p$. We can write $Z_1 = \frac{1}{s} UV$, where $U = \scpr{\vec \omega_j, \vec x}$ and $V = \scpr{\vec \omega_j, \vec y}$. Let us first assume that $\vec x \ne \vec y.$ In this case, the random pair $(U, V)$ is a bivariate Gaussian such that 
\begin{equation*}
\expec[U] = \expec[V]=0,\quad
\Var(U) = \Var(V) = 1,\quad
\Cov(U,V) = \scpr{\vec x, \vec y},
\end{equation*}
and it follows from \cite[Theorem~2.1]{NadarajahPogany16} that the product $UV$ has probability density function
\begin{equation*}
	f_{UV}(z) := \frac{1}{\pi \sqrt{1 - \scpr{\vec x, \vec y }^2}} \exp\left(\frac{\scpr{\vec x, \vec y} z}{1 - \scpr{\vec x, \vec y}^2} \right) K_0 \left(\frac{\abs{z}}{1 - \scpr{\vec x, \vec y}^2}\right),
\end{equation*}
where $K_0$ denotes the modified Bessel function of the second kind of order zero, which is given by (see \cite[eq.~(9.6.24)]{AbramowitzStegun64})
\begin{equation*}
	\Kzero(u)=\int_{0}^{\infty}e^{-u\cosh t}\,dt, \qquad u > 0.
\end{equation*} 
We have
\begin{align*}
\|f_{UV}\|_{p}^p = \int_{-\infty}^{\infty}f_{UV}(z)^p\,dz &= \frac{1}{\pi^p\,(1 - \scpr{\vec x, \vec y}^2)^{p/2}} \int_{-\infty}^\infty \exp\left(\frac{p \scpr{\vec x, \vec y} z}{1 - \scpr{\vec x, \vec y}^2} \right)\,
K_0 \left(\frac{\abs{z}}{1 - \scpr{\vec x, \vec y}^2}\right)^p dz \\
&= \frac{1}{\pi^p (1 - \scpr{\vec x, \vec y}^2)^{p/2 - 1}} \int_{-\infty}^{\infty} e^{p u \scpr{\vec x, \vec y}} K_0(\abs{u})^p \, du \\
&\le \frac{2}{\pi^p} (1 - \scpr{\vec x, \vec y}^2 )^{1-p/2} \int_{0}^{\infty} e^{ p u \abs{\scpr{\vec x, \vec y}}} K_0(u)^p \, du,
\end{align*}
where we used the change of variables $u = z / (1 - \scpr{\vec x, \vec y}^2)$. 
We bound this integral by splitting it on the two intervals $[0, 1]$ and $[1, \infty)$. 
For the integral on $[1, \infty)$, we have $\cosh t \ge 1 + \frac{1}{2}t^2$, so for all $u \ge 1$ we have the bound
\begin{equation*}
	K_0(u) = \int_0^\infty e^{-u\cosh t}\,dt
\le \int_0^\infty e^{-u -\frac{1}{2}u t^2}\,dt
=e^{-u}\sqrt{\frac{\pi}{2u}}.
\end{equation*}
Defining $\theta := 1 - \abs{\scpr{\vec x, \vec y}}$, we obtain
\begin{equation*}
	\int_{1}^{\infty} e^{pu \abs{\scpr{\vec x, \vec y}}} K_0(u)^p \, du \le 
	\int_{1}^{\infty} e^{p u \abs{\scpr{\vec x, \vec y}}} e^{-pu} \left(\frac{\pi}{2u}\right)^{p/2}  \,du \le \left(\frac{\pi}{2}\right)^{p/2} \int_{1}^{\infty} e^{-p \theta u} u^{-p/2}\,du.
\end{equation*}
With the change of variables $v = p\theta u$, we have
\begin{align*}
	\int_{1}^{\infty} e^{pu \abs{\scpr{\vec x, \vec y}}} K_0(u)^p \, du &= \left(\frac{\pi}{2}\right)^{p/2} (p \theta)^{p/2 - 1} \int_{p \theta}^\infty e^{-v} v^{-p/2} \, dv \\
	&\le \left(\frac{\pi}{2}\right)^{p/2} (p \theta)^{p/2 - 1} \int_{0}^\infty e^{-v} v^{-p/2} \, dv \\
	&= \left(\frac{\pi}{2}\right)^{p/2} (p \theta)^{p/2 - 1} \Gamma(1 - p/2).
\end{align*}
Note that for the last identity we used $p < 2$, so that $1 - p/2 > 0$ and $\Gamma(1-p/2)$ is well defined.   

For the integral on $[0, 1]$, using $e^{p u \abs{\scpr{\vec x, \vec y}}} \le e^{p \abs{\scpr{\vec x, \vec y}}}$ and \cref{lemma:k0-bound-near-zero} to bound $K_0(u) \le K_0(1) - \ln u$, we have
\begin{equation*}
	\int_0^1 e^{pu \abs{\scpr{\vec x, \vec y}}} K_0(u)^p du \le e^{p \abs{\scpr{\vec x, \vec y}}} \int_0^1 (K_0(1) - \ln u)^p du \le e^{p \abs{\scpr{\vec x, \vec y}}} \int_0^1 (1 - \ln u)^2 =  5 e^{p \abs{\scpr{\vec x, \vec y}}},
\end{equation*}
where we used the facts that $K_0(1) \approx 0.421 \le 1$ and $(1+x)^p \le (1+x)^2$ for all $x \ge 0$ and $p \in (1, 2)$.  
Putting together the two parts of the integral, we get the bound
\begin{equation*}
	\norm{f_{UV}}_p^p \le \frac{2}{\pi^p} (1 - \scpr{\vec x, \vec y}^2 )^{1-p/2} \left( 5 e^{p \abs{\scpr{\vec x, \vec y}}} + \left(\frac{\pi}{2}\right)^{p/2} (p \theta)^{p/2 - 1} \Gamma(1 - p/2) \right).
\end{equation*}
Using $1 - \scpr{\vec x, \vec y}^2 = \theta(2 - \theta)$ and some simple inequalities to obtain a bound independent of $\scpr{\vec x, \vec y}$, we get
\begin{equation*}
	\norm{f_{UV}}_p^p \le 10 \left(\frac{e}{\pi}\right)^p + \frac{2^{2-p}}{\pi^{p/2}} p^{p/2 - 1} \Gamma(1 - p/2).
\end{equation*}
Fixing $k = 3$ we have $p = 3/2$, and the bound becomes
\begin{equation*}
	\norm{f_{UV}}_{3/2}^{3/2} \le 10 \left(\frac{e}{\pi}\right)^{3/2} + \frac{\sqrt{2}}{\pi^{3/4}} \left(\frac{3}{2}\right)^{-1/4} \Gamma(1/4) \le 10.012.
\end{equation*}
Recalling that $Z_1 = \frac{1}{s} UV$, we have $f_{Z_1}(z) = s f_{UV}(sz)$ and hence $\norm{f_{Z_1}}_{p}^p = s^{p-1}\norm{f_{UV}}_{p}^p$. Combined with \cref{eqn:proof-sketch-ip--convolution-inequality}, we obtain
\begin{equation}
	\label{eqn:proof-sketch-ip--linftynorm-bound}
		\norm{f_{Z}}_\infty \le s \norm{f_{UV}}_{3/2}^3 \le 101 \, s.
\end{equation}
Finally, combining \cref{eqn:proof-sketch-ip--prob-infnorm-bound} with \cref{eqn:proof-sketch-ip--linftynorm-bound} we get
\begin{equation*}
	\prob (\abs{\scpr{\sketch \vec x, \sketch \vec y}} \le \delta) \le 202 \, s \delta,
\end{equation*}
thus concluding the proof for the case $\vec x \ne \vec y$.

Assume now $\vec x = \vec y$. We have 
\begin{equation*}
	Z = \scpr{\sketch \vec x, \sketch \vec x} = \frac{1}{s} \sum_{j = 1}^s \scpr{\vec \omega_j, \vec x}^2,
\end{equation*}
and since $\scpr{\vec \omega_j, \vec x}$ are i.i.d.~$\mathcal{N}(0,1)$ random variables, we have that $Z = \frac{1}{s} W$, where $W$ is a $\chi^2$ random variable with $s$ degrees of freedom. Therefore
\begin{equation*}
	\prob (\abs{\scpr{\sketch \vec x, \sketch \vec x}} \le \delta) = \prob (W \le s \delta) \le (\delta e^{1 - \delta})^{s/2},
\end{equation*}
where the last inequality is a Chernoff bound on the left tail of $W$, see for instance the proof of \cite[Lemma~3.2]{BRS24}. 
\end{proof}

\begin{remark}
	\label{rem:proof-comments}
	In the proof of \cref{prop:sketched-inner-product--gaussian}, we used $p = 3/2$ instead of $p = 2$ to obtain a bound that does not deteriorate as $\scpr{\vec x, \vec y} \to 1$. Our goal was mainly to show that $\prob( \scpr{\sketch \vec x, \sketch \vec y} \le \delta ) = \mathcal{O}(\delta)$, so we did not strive to obtain the sharpest possible inequalities in the proof, and it is very likely that the constant factor in the bound and its dependence on $s$ can be improved.
\end{remark}

\begin{remark}
	\label{rem:epsilon-embedding-of-orthogonal-vectors}
	
	In particular, \cref{prop:sketched-inner-product--gaussian} shows that even if $\vec x$ and $\vec y$ are orthogonal, their sketches $\sketch \vec x$ and $\sketch \vec y$ are very unlikely to be close to orthogonal. At first glance, this might seem to contradict the $\varepsilon$-embedding property of $\sketch$, since the sketching matrix should approximately preserve inner products. However, there is no contradiction because inner products are only preserved in an absolute sense, rather than relative. Indeed, if $\scpr{\vec x, \vec y} = 0$, the $\varepsilon$-embedding property \cref{eqn:epsilon-embedding-property} only tells us that 
	\begin{equation*}
		\abs{\scpr{\sketch \vec x, \sketch \vec y}} \le \varepsilon \norm{\vec x} \norm{\vec y},
	\end{equation*}   
	which does not contradict \cref{prop:sketched-inner-product--gaussian} since the $\varepsilon$ used for subspace embeddings is usually relatively large, say $\varepsilon = 1/2$ or $\varepsilon = 1/\sqrt{2}$.
	
\end{remark}

\cref{prop:sketched-inner-product--gaussian} allows us to gain some understanding on the improved stability of the randomized two-sided Gram-Schmidt process with respect to the standard one. Indeed, the ill-conditioning of the bases $\mat Q$ and $\mat P$ mainly stems from the fact that the biorthogonal vectors $\vec q_i$ and $\vec p_i$ often have very large norms, or equivalently that the normalized vectors $\vec q_i/\norm{\vec q_i}$ and $\vec p_i / \norm{\vec p_i}$ have small inner product. However, \cref{prop:sketched-inner-product--gaussian} shows that, for a Gaussian sketching matrix $\sketch$, the sketched inner product of any two normalized vectors $\vec q$ and $\vec p$ only has probability $O(\delta)$ of being smaller than $\delta$ (this probability is even smaller if $\vec q = \vec p$). Hence, it is extremely unlikely that \cref{algorithm:two-sided-gram-schmidt-sketched} constructs sketch-biorthogonal vectors with very large norms, and the condition number of the sketch-biorthogonal bases $\mat Q$ and $\mat P$ is very unlikely to become too large.  

In \cref{fig:orth_vs_sketch_orth} we illustrate that the property that we proved for a Gaussian sketch in \cref{prop:sketched-inner-product--gaussian} also holds numerically for other commonly used sketching matrices. Specifically, for two unit random vectors $\vec x$ and $\vec y$ such that $\scpr{\vec x, \vec y}=0$, we plot $\scpr{\sketch \vec x, \sketch \vec y}$ for different sketching matrices $\sketch \in \R^{s \times n}$ and different values of $s$, showing both the average value and the minimum over $100$ tests. For all sketching matrices, we observe that it is highly unlikely to obtain a sketched inner product that is close to zero. In light of this experiment, we expect that results analogous to \cref{prop:sketched-inner-product--gaussian} could also be proved for other kinds of sketching matrices.
In \cref{sec:numerical-experiments}, we are going to display the correlation between the inner product of $\vec q_i / \norm{\vec q_i}$ and $\vec p_i / \norm{\vec p_i}$ and the growth of the condition number of $\mat Q_i$ and $\mat P_i$, which combined with the observations in this section provides a theoretical justification for the better stability of the randomized two-sided Gram-Schmidt process.

\begin{figure}[tb]
	\centering
    \includegraphics[width=0.55\textwidth]{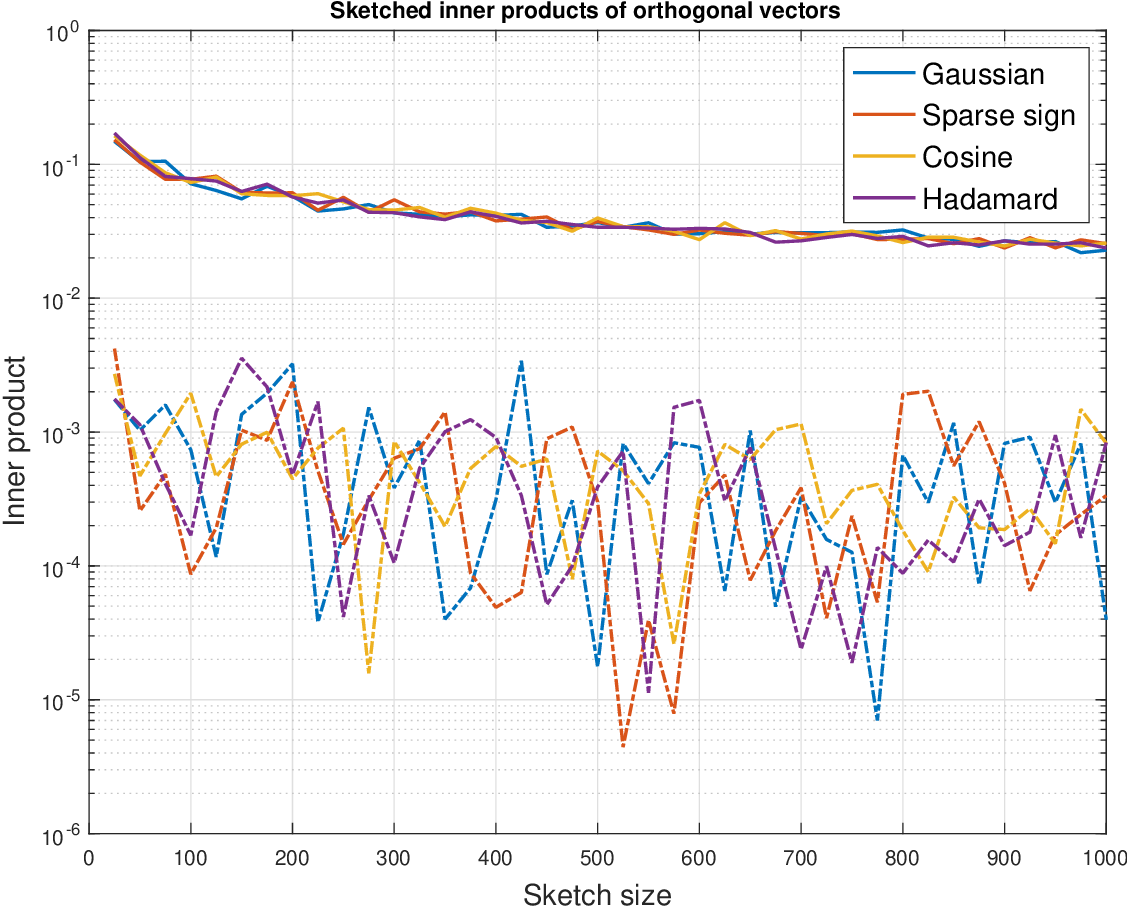}
	\caption{Sketched inner product $\scpr{\sketch \vec x, \sketch \vec y}$ of random orthogonal unit vectors $\vec x, \vec y \in \R^n$ of dimension $n = 10^4$, for different sketching sizes $s=25:25:10^3$ and different sketching matrices $\sketch \in \R^{s \times n}$. Solid: average over $100$ tests. Dash-dotted: minimum over $100$ tests.
    }
	\label{fig:orth_vs_sketch_orth}
\end{figure}

\section{Application to nonsymmetric Lanczos}
\label{sec:application-to-nonsymmetric-lanczos}

In this section, we show how \cref{algorithm:two-sided-gram-schmidt-sketched} can be incorporated within the nonsymmetric Lanczos algorithm and applied for the computation of eigenvalues and eigenvectors of a matrix $\mat A \in \R^{n \times n}$. We first recall some basic notions about Krylov subspaces \cite{Saad89, Arnoldi51}. 

\subsection{The nonsymmetric Lanczos algorithm}
\label{subsec:nonsym-lanczos}

Given a matrix $\mat A\in \R^{n\times n}$ and a vector $\vec b\in \R^n$, we denote with
\begin{equation*}
    \kryl_m(\mat A, \vec b) = {\rm span}\{ \vec b, \mat A\vec b, \mat A^2 \vec b, \ldots, \mat A^{m-1}\vec b\} = \{ p(\mat A) \vec b \; : \; \deg p \le m-1 \}, 
\end{equation*}
the Krylov space generated by $\mat A$ and $\vec b$.
The nonsymmetric Lanczos algorithm \cite{Lanczos50} is a generalization of the Lanczos algorithm (that can be seen a simplification of the Arnoldi method to symmetric matrices, see e.g. \cite{Arnoldi51}) to nonsymmetric matrices $\mat A\in\bb{R}^{n\times n}$, which uses a short-term recurrence to iteratively build the two Krylov subspaces
\begin{align*}
    \kryl_m(\mat A, \vec q_1) &= {\rm span}\{\vec q_1, \mat A \vec q_1, \ldots, \mat A^{m-1} \vec q_1\}, \\ 
    \kryl_m(\mat A^T, \vec p_1) &= {\rm span}\{\vec p_1, \mat A \vec p_1, \ldots, \mat A^{m-1} \vec p_1\}.
\end{align*}
The algorithm constructs two bases, denoted by $\mat Q_m = [\vec q_1, \vec q_2, \ldots, \vec q_m]$ and $\mat P_m = [\vec p_1, \vec p_2, \ldots, \vec p_m]$, respectively spanning $\kryl_m(\mat A, \vec q_1)$ and $\kryl_m(\mat A^T, \vec p_1)$, that are biorthogonal, i.e.~$\mat Q_m^T \mat P_m = I$. The bases satisfy the Arnoldi relations
\begin{equation}
	\label{eqn:nonsym-lanczos-relations}
	\begin{aligned}
		\mat A \mat Q_m &= \mat Q_m \mat H_m + \delta_{m+1} \vec q_{m+1} \vec e_m^T,\\
		\mat A^T \mat P_m &= \mat P_m \mat T_m + \beta_{m+1} \vec p_{m+1} \vec e_m^T,
	\end{aligned}
\end{equation}
	where $\mat H_m, \mat T_m\in\bb{R}^{m\times m}$ are upper Hessenberg matrices. From these relations we obtain 
\begin{equation}
	\label{eqn:nonsym-lanczos-projected}
	\begin{aligned}
		\mat P_m^T \mat A \mat Q_m &= \mat H_m, \\
		\mat Q_m^T \mat A^T \mat P_m &= \mat T_m,
	\end{aligned}
\end{equation}
and given that the left-hand sides in \cref{eqn:nonsym-lanczos-projected} are one the transpose of the other, we have $\mat H_m = \mat T_m^T$. This implies that $\mat H_m$ and $\mat T_m$ are both tridiagonal matrices, showing that the bases $\mat Q_m$ and $\mat P_m$ can be constructed with a short-term recurrence, see e.g. \cite{PTL85, Saad11}, similarly to the Lanczos algorithm for symmetric matrices. 

It is well known that the nonsymmetric Lanczos algorithm suffers from \textit{breakdown} or \textit{near breakdown} situations, which mainly occur when the inner product between the newly computed basis vectors $\vec q_i$ and $\vec p_i$ is (numerically) zero. 
In the literature, different ways of overcoming breakdowns have been proposed. The most accredited technique is the look-ahead Lanczos algorithm \cite{FGM93, Parlett92}, where we have $\mat Q_m^T \mat P_m = \mat D_m$, with a matrix $\mat D_m$ which is block diagonal instead of diagonal. This may be useful in avoiding some breakdowns, but even this approach is not sufficient in all cases.

\subsection{A randomized variant of nonsymmetric Lanczos}
\label{subsec:randomized-nonsym-lanczos}

We now introduce a randomized variant of the nonsymmetric Lanczos algorithm, which uses the randomized two-sided Gram-Schmidt process to construct sketch-biorthogonal bases of the two Krylov subspaces $\kryl_m(\mat A, \vec q_1)$ and $\kryl_m(\mat A^T, \vec p_1)$. 
Specifically, given a sketching matrix $\sketch \in \R^{s \times n}$, the randomized nonsymmetric Lanczos algorithm constructs two bases $\mat Q_m$ and $\mat P_m$, respectively spanning $\kryl_m(\mat A, \vec q_1)$ and $\kryl_m(\mat A^T, \vec p_1)$, such that $(\sketch \mat Q_m)^T \sketch \mat P_m = I$. 
The algorithm also constructs upper Hessenberg matrices $\mat H_m$ and $\mat T_m$ that still satisfy the relations \cref{eqn:nonsym-lanczos-relations}, and we also have the sketched relations  
\begin{equation}
	\label{eqn:sketched-nonsym-lanczos}
	\begin{aligned}
		\sketch \mat A \mat Q_m &= \sketch \mat Q_m \mat H_m + \delta_{m+1} \sketch \vec q_{m+1} \vec e_m^T,\\
		\sketch \mat A^T \mat P_m &= \sketch \mat P_m \mat T_m + \beta_{m+1} \sketch \vec p_{m+1} \vec e_m^T.
	\end{aligned}
\end{equation}
Since now the bases are sketch-biorthogonal, from \cref{eqn:sketched-nonsym-lanczos} we obtain
\begin{equation}
	\label{eqn:sketched-nonsym-lanczos-projection}
	\begin{aligned}
		(\sketch \mat P_m)^T \sketch \mat A \mat Q_m &= \mat H_m, \\
		(\sketch \mat Q_m)^T \sketch \mat A^T \mat P_m &= \mat T_m.
	\end{aligned}
\end{equation}
Note that \cref{eqn:sketched-nonsym-lanczos-projection} replaces \cref{eqn:nonsym-lanczos-projected}, and in this case the left-hand sides are no longer always the transpose of one another. Therefore the matrices $\mat H_m$ and $\mat T_m$ are in general not tridiagonal, showing that in this randomized variant the sketch-biorthogonal bases cannot be constructed with a short-term recurrence.

Even though the sketched inner products used in the randomized nonsymmetric Lanczos algorithm are significantly cheaper than the ones involving vectors of length $n$ used in its standard counterpart, the standard nonsymmetric Lanczos algorithm is still overall cheaper when its implementation exploits the short-term recurrence, since in its $m$-th iteration we only have to orthogonalize $\vec q_{m+1}$ against $\vec p_m$ and $\vec p_{m-1}$, instead of against the whole basis $\mat P_m$.
On the other hand, in finite precision arithmetic the short-term recurrence usually suffers from a loss of biorthogonality, and thus explicit re-biorthogonalization of the full bases is often required in applications \cite{MorganNicely11, JaimoukhaKasenally95}. This is especially true in applications in which the Lanczos algorithm is used for computing eigenvalues, where the loss of orthogonality causes the appearance of so called ``ghost eigenvalues'' even in the symmetric case \cite{ParlettScott79,CullumWilloughby02}, and hence it is crucial to fully biorthogonalize the bases explicitly.
In such a setting, the standard nonsymmetric Lanczos algorithm loses the computational advantage given by the short-term recurrence, and the randomized variant introduced in this section becomes more competitive. We can also expect that the randomized variant will construct bases that are often better conditioned compared to the standard algorithm, as a consequence of the analysis in \cref{sec:theoretical-analysis}.

The following result establishes an optimality property of the characteristic polynomial of the Hessenberg matrix $\mat H_m$ constructed by the randomized nonsymmetric Lanczos algorithm, which draws a connection with the original nonsymmetric Lanczos algorithm.

\begin{proposition}
	\label{prop:charpoly-optimality-sketched-nonsym-lanczos}
	Let $\mat H_m$ be as defined in \cref{eqn:sketched-nonsym-lanczos}, and let $p_{\mat H_m}$ be its characteristic polynomial. Denoting by $\Pi_m$ the set of monic polynomials of degree less than or equal to $m$, we have 
	\begin{equation*}
		p_{\mat H_m} = \argmin_{\vec q \in \Pi_m} \, \norm{(\sketch \mat P_m)^T \sketch q(\mat A) \vec b}.
	\end{equation*}
\end{proposition}
\begin{proof}
	The proof follows the same approach used in the proofs of \cite[Theorem~6.1]{Saad11} and \cite[Theorem~5.2]{DeDamasGrigori24}.
	Let us denote by $\skproj{Q} = \mat Q_m (\sketch \mat P_m)^T \sketch$ the sketched oblique projector onto $\kryl_m(\mat A, \vec b)$, acting sketch-orthogonally with respect to $\kryl_m(\mat A^T, \vec c)$. We have 
	\begin{equation*}
		\skproj{Q} \mat A \skproj{Q} = \mat Q_m (\sketch \mat P_m)^T \sketch \mat A \mat Q_m (\sketch \mat P_m)^T \sketch = \mat Q_m \mat H_m (\sketch \mat P_m)^T \sketch.
	\end{equation*}
	Let us denote by $p_{\mat M}$ the characteristic polynomial of a matrix $\mat M$. Given two matrices $\mat M \in \R^{n \times m}$ and $\mat N \in \R^{m \times n}$ with $m \le n$, it is easy to see that $p_{\mat M \mat N}(z) = p_{\mat N \mat M}(z) z^{n - m}$. In particular, using $\mat M = \mat Q_m \mat H_m$ and $\mat N = (\sketch \mat P_m)^T \sketch$, we have $p_{\skproj{Q} \mat A \skproj{Q}}(z) = p_{\mat H_m} (z) z^{n-m}$.  
	By the Cayley-Hamilton theorem, we have $p_{\skproj{Q} \mat A \skproj{Q}} (\skproj{Q} \mat A \skproj{Q}) = \mat 0$. Assuming that $\mat H_m$ is full rank, we have that $\skproj{Q} \mat A \skproj{Q}$ is invertible on $\kryl_m(\mat A, \vec b)$, and hence for any $\vec x \in \kryl_m(\mat A, \vec b)$ we have 
	\begin{equation*}
		p_{\mat H_m}(\skproj{Q} \mat A \skproj{Q}) \vec x = \vec 0.
	\end{equation*}
	By \cite[Proposition~6.4]{Saad11}, since $\deg p_{\mat H_m} \le m$ and $\skproj{Q}$ is a projector onto $\kryl_m(\mat A, \vec b)$, we have   
	\begin{equation*}
		\skproj{Q} p_{\mat H_m}(\mat A) \vec b = p_{\mat H_m}(\skproj{Q} \mat A \skproj{Q}) \vec b = \vec 0.
	\end{equation*}
	This implies that 
	\begin{equation*}
		\scpr{\sketch \skproj{Q} p_{\mat H_m}(\mat A) \vec b, \sketch \vec y} = 0 \qquad \forall \, \vec y \in \R^n,
	\end{equation*}
	and by \cref{prop:sketched-oblique-projector-properties} 
	\begin{equation*}
		\scpr{\sketch p_{\mat H_m}(\mat A) \vec b, \sketch \skproj{P} \vec y} = 0,
	\end{equation*}
	where $\skproj{P}$ denotes the sketched oblique projector onto $\kryl_m(\mat A^T, \vec c)$ acting sketch-orthogonally with respect to $\kryl_m(\mat A, \vec b)$. 
	Since $p_{\mat H_m}$ is a monic polynomial, we can write it as $p_{\mat H_m}(x) = x^m - s(x)$, where $\deg s \le m-1$. We have shown that $\mat A^m \vec b - s(\mat A) \vec b \perp_\sketch \kryl_m(\mat A^T, \vec c)$, and since $s(\mat A) \vec b \in \kryl_m(\mat A, \vec b)$, we conclude that $s(\mat A) \vec b = \skproj{Q} A^m \vec b$. Therefore, by \cref{prop:sketched-oblique-projector-properties} we have
	\begin{equation*}
		s(\mat A) \vec b = \argmin_{\vec y \in \kryl_m(\mat A, \vec b)} \norm{(\sketch \mat P_m)^T \sketch (\mat A^m \vec b - \vec y)} = \argmin_{q \,:\, \deg q \le m-1} \norm{(\sketch \mat P_m)^T \sketch (\mat A^m \vec b - q(\mat A) \vec b)},
	\end{equation*}       
	and this concludes the proof because any monic polynomial in $\Pi_m$ can be written as $\mat A^m - q(\mat A)$ for some polynomial $q$ with $\deg q \le m-1$.
\end{proof}

\begin{remark}
	\label{rem:charpoly-optimality-transpose}
	With the same proof as \cref{prop:charpoly-optimality-sketched-nonsym-lanczos}, it can be shown that the characteristic polynomial of $\mat T_m$ satisfies the optimality condition
	\begin{equation*}
		p_{\mat T_m} = \argmin_{q \in \Pi_m} \norm{(\sketch \mat Q_m)^T \sketch q(\mat A^T) \vec c}.
	\end{equation*} 
	Note that, in contrast with the classical nonsymmetric Lanczos algorithm, we have $\mat H_m \ne \mat T_m^T$ and hence the two characteristic polynomials $p_{\mat H_m}$ and $p_{\mat T_m}$ do not necessarily coincide. 
\end{remark}

\subsection{Approximation of eigenvalues with nonsymmetric Lanczos} 
\label{subsec:nonsym-lanczos-for-eigenvalues}

The nonsymmetric Lanczos method can be used to approximate eigenvalues and eigenvectors of a possibly nonsymmetric matrix $\mat A\in\bb{R}^{n\times n}$. After computing the biorthogonal bases $\mat Q_m$, $\mat P_m$ and the tridiagonal matrix $\mat H_m$, the eigenvalues of $A$ are approximated by solving a smaller eigenvalue problem with $\mat H_m$. More precisely, given an eigentriplet of $\mat H_m$, say $(\theta_i, \tilde{\vec x}_i, \tilde{\vec y}_i)$ we can construct the Ritz triplet of $\mat A$ as $(\theta_i, \vec x_i, \vec y_i) = (\theta_i, \mat Q_m \tilde{\vec x}_i, \mat P_m \tilde{\vec y}_i)$. 
The perturbation theory for the nonsymmetric case is more complicated than in the symmetric case. In \cite{KPJ82, Saad11}, theoretical bounds on the error between the eigenvalues of $\mat H_m$ and those of $\mat A$ are given, showing that the approximate eigentriplets computed by the nonsymmetric Lanczos method solve a nearby eigenvalue problem.

We propose to use the randomized nonsymmetric Lanczos algorithm to compute the two sketch-biorthogonal bases $\mat Q_m$ and $\mat P_m$ satisfying \cref{eqn:sketched-nonsym-lanczos} and \cref{eqn:sketched-nonsym-lanczos-projection}, and compute the eigenvalues and eigenvectors of the associated upper Hessenberg matrices $\mat H_m$ and $\mat T_m$. As we already discussed in \cref{subsec:randomized-nonsym-lanczos}, to avoid the occurrence of ghost eigenvalues it is beneficial to fully biorthogonalize the bases even when a short-term recurrence is available, so the randomized algorithm is more efficient than the standard nonsymmetric Lanczos method, even though the short-term recurrence is lost in the randomized case. 
However, the main drawback of the randomized approach is that the two matrices $\mat H_m$ and $\mat T_m$ in \cref{eqn:sketched-nonsym-lanczos-projection} are in general not similar to each other, and hence their eigenvalues are not necessarily the same. Therefore, in contrast with the standard nonsymmetric Lanczos algorithm, we do not directly find approximate eigentriplets of $\mat A$. Nevertheless, we are going to see in \cref{subsec:nonsymm_numerical_example} that the residuals of the approximate eigenvalues and eigenvectors computed by the randomized nonsymmetric Lanczos method are comparable with those of the deterministic method.

\section{Numerical experiments}
\label{sec:numerical-experiments}

In this section we present some experiments that illustrate the numerical behavior of the randomized two-sided Gram-Schmidt algorithm and its application to the nonsymmetric Lanczos method for computing eigenvalues. All experiments have been performed in MATLAB 2025a, on a Dell laptop with CPU 12th Gen Intel(R) Core(TM) i7-12700H 2.30GHz and 16 GB RAM.

\subsection{Condition number growth and loss of biorthogonality}
\label{subsec:experiment-condition-number-growth}

We begin with a comparison of different implementations of the randomized two-sided Gram-Schmidt algorithm (\cref{algorithm:two-sided-gram-schmidt-sketched}) and its deterministic counterpart (\cref{algorithm:two-sided-gram-schmidt}), focusing on the growth of the condition number of the computed bases $\mat Q_m$ and $\mat P_m$, and the numerical loss of biorthogonality $\norm{I - \mat P_m^T \mat Q_m}_F$ or sketch-biorthogonality $\norm{I - (\sketch \mat P_m)^T \sketch \mat Q_m}_F$. We compare modified Gram-Schmidt (MGS and MGS2), classical Gram-Schmidt (CGS, CGS2 and CGS3), and classical Gram-Schmidt with explicit oblique projection (CGS\_O and CGS\_O2) against their randomized versions, for which we use the same abbreviations with an intial r, which stands for ``randomized''.

\subsubsection{Ill-conditioned matrices}
\label{subsubsec:experiment-ill-conditioned}
We first consider two highly ill-conditioned bases $\mat X$ and $\mat Y \in \R^{n \times m}$, with dimensions $n=10^4$ and $m = 200$, built by discretizing the functions
\[
f(x,y) = \frac{\sin(x+y)}{\cos(100\cdot (y-x)) + 1.1}, \hspace{1cm}
g(x,y) = \frac{\cos(x+y)}{\sin(200\cdot (y-x)) + 1.2},
\]
on a uniform grid on $[0, 1] \times [0, 1]$, i.e.~we set $\mat X_{i,j} = f(\frac{i-1}{n-1},\frac{j-1}{m-1})$ and $\mat Y_{i,j} = g(\frac{i-1}{n-1},\frac{j-1}{m-1})$. The condition number of these matrices grows rapidly with $m$: for $m = 200$, we have $\kappa(\mat X) = 3.88 \times 10^{15}$ and $\kappa(\mat Y) = 4.13\times 10^{15}$. These synthetic test matrices are adapted from \cite{BalabanovGrigori22}, with some minor modifications.

\begin{table}
	\caption{Performance of two-sided Gram-Schmidt and randomized two-sided Gram-Schmidt on ill-conditioned $\mat X$ and $\mat Y \in \R^{n \times m}$, with $n = 10^4$ and $m = 200$. 
	\label{tab:exp1}}
\centering
\begin{tabular}{l|cccccc}
\toprule
 & time (s) & ${\rm cond}(\mat Q)$ & ${\rm cond}(\mat P)$ & $\text{err}(\mat X)$ & $\text{err}(\mat Y)$ & (sketch-)biorth\\
\midrule
MGS & 0.606 & 2.108e+10 & 2.165e+11 & 2.430e-12 & 7.861e-11 & 1.147e+02 \\

MGS2 & 1.090 & 8.013e+09 & 5.737e+10 & 2.846e-12 & 9.716e-11 & 1.518e-06 \\

CGS & 0.229 & 1.842e+17 & 8.317e+16 & 2.425e-12 & 7.520e-10 & 6.513e+02 \\

CGS2 & 0.394 & 1.960e+17 & 6.260e+17 & 1.398e-09 & 7.498e-08 & 4.888e+02 \\

CGS3 & 0.554 & 1.150e+11 & 1.648e+12 & 2.575e-12 & 8.133e-11 & 7.391e-06 \\

CGS\_O & 0.355 & 8.332e+17 & 5.497e+17 & 2.081e-12 & 2.476e-12 & 1.784e+03 \\

CGS\_O2 & 0.515 & 4.114e+09 & 5.908e+10 & 2.458e-12 & 8.139e-11 & 5.699e-03 \\

\midrule

rMGS & 0.397 & 2.714e+05 & 9.587e+05 & 4.882e-12 & 3.234e-11 & 1.077e+02 \\

rMGS2 & 0.659 & 1.333e+05 & 5.699e+05 & 5.999e-12 & 3.980e-11 & 2.527e-11 \\

rCGS & 0.161 & 1.290e+17 & 2.044e+17 & 7.692e-12 & 2.382e-10 & 6.340e+02 \\

rCGS2 & 0.278 & 8.429e+16 & 1.246e+17 & 8.147e-10 & 1.318e-08 & 1.324e+02 \\

rCGS3 & 0.419 & 3.107e+05 & 9.504e+05 & 5.215e-12 & 3.308e-11 & 3.050e-11 \\

rCGS\_O & 0.197 & 7.906e+16 & 4.294e+16 & 7.050e-13 & 6.862e-12 & 6.839e+01 \\

rCGS\_O2 & 0.320 & 1.639e+05 & 7.254e+05 & 4.943e-12 & 3.259e-11 & 9.432e-10 \\
\bottomrule
    \end{tabular}

\end{table}

In \cref{tab:exp1} we report the results both for all the implementations of the deterministic and randomized two-sided Gram-Schmidt processes. In \cref{tab:exp1}, $\text{err}(\mat X) = \norm{\mat X - \mat Q \mat T}_F$ and $\text{err}(\mat Y) = \norm{\mat Y - \mat P \mat S}_F$ denote the errors in the biorthogonal decompositions computed numerically, while the (sketch-)biorth column contains the numerical loss of biorthogonality for the deterministic algorithms, and the loss of sketch-biorthogonality for the randomized ones. We observe that, when a deterministic algorithm is unable to achieve a small biorthogonalization error (see for instance MGS, CGS and CGS2), then the same is usually true for the corresponding randomized version. On the other hand, the advantages of the randomized version are evident for MGS2, CGS3 and CGS\_O2: not only the biorthogonalization error decreases of several orders of magnitude, but the condition number of both bases $\mat Q$ and $\mat P$ are also significantly lower. Moreover, in all cases the randomized version requires a significant lower computational time compared to the deterministic ones. In this experiment with highly ill-conditioned matrices, the method that overall seems to perform the best is rCGS\_O2. Although it achieves similar accuracy to rMGS2 and rCGS3, its shorter runtime makes it the preferred algorithm. 

It is also interesting to analyze the condition number growth and the increasing loss of biorthogonality for the different algorithms as iterations progress. In \cref{fig:biorth_loss} we plot the loss of biorthogonality for all methods during the first $200$ iterations using the same example. 
In almost all the cases, the randomized algorithms obtain lower (sketch-)biorthogonalization error than their deterministic counterpart. 

\begin{figure}[tb]
	\centering
    \begin{subfigure}{0.325\textwidth}
    \includegraphics[width=\textwidth]{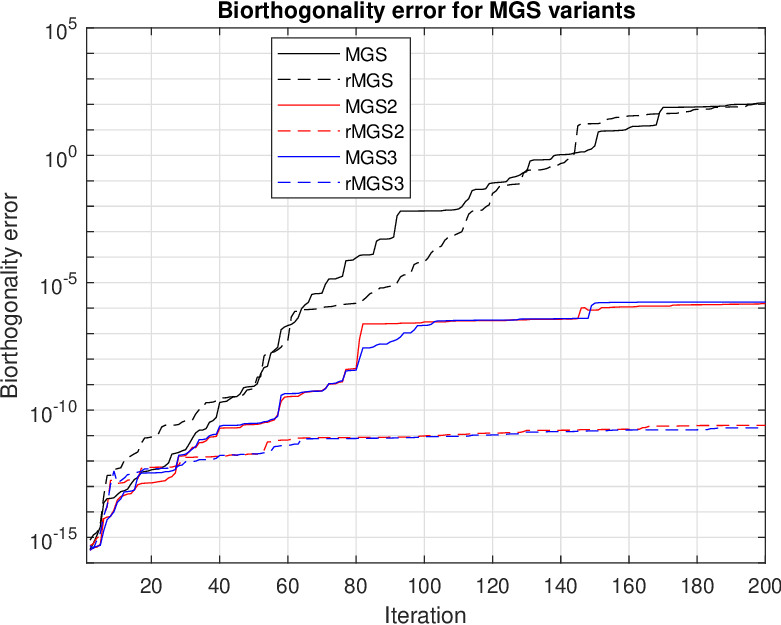}
    \caption{MGS variants}
    \label{fig:biorth_err_mgs_var}
    \end{subfigure}
    \begin{subfigure}{0.325\textwidth}
    \includegraphics[width=\textwidth]{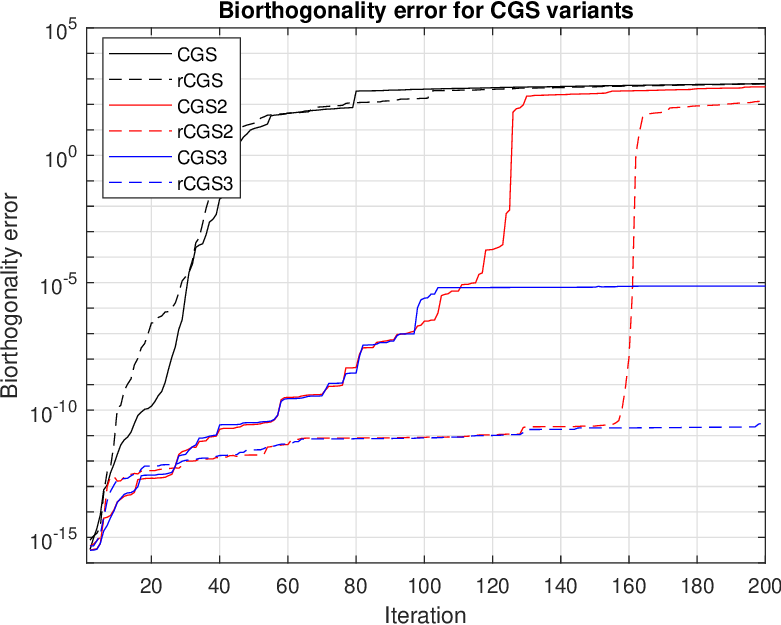}
    \caption{CGS variants}
    \label{fig:biorth_err_cgs_var}
    \end{subfigure}
    \begin{subfigure}{0.325\textwidth}
    \includegraphics[width=\textwidth]{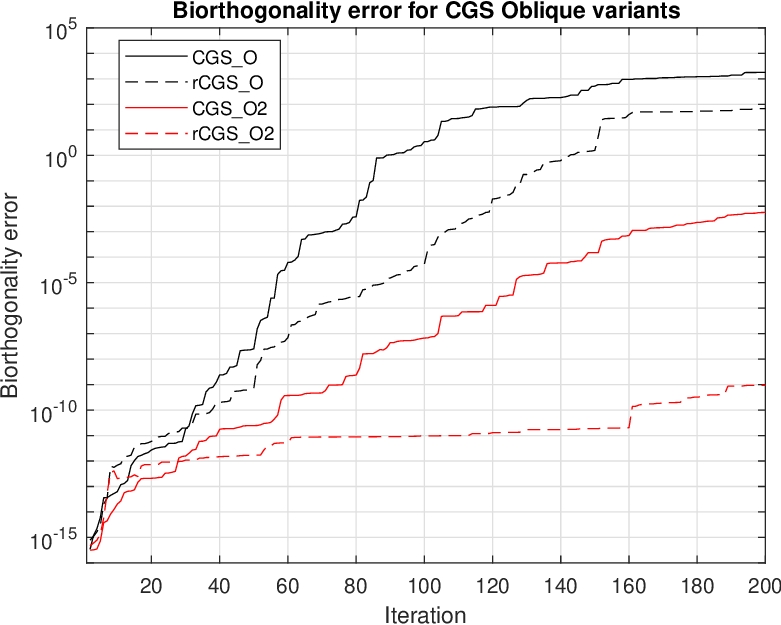}
    \caption{CGS\_O variants}
    \label{fig:biorth_err_cgs_o_var}
    \end{subfigure}
	\caption{Biorthogonality errors computed at each iteration, for ill conditioned $\mat X$ and $\mat Y$, as in \cref{tab:exp1}. (a): MGS and its variants. (b): CGS and its variants. (c) CGS\_O and its variants.
    }
	\label{fig:biorth_loss}
\end{figure}

To deepen our understanding of the loss of biorthogonality and the condition number growth, we now focus on two implementations, namely CGS\_O2 and rCGS\_O2. 
Recall that the cosine of the angle between two vectors $\vec x,\vec y\in\bb{R}^{n}$ is defined as
 \begin{equation*}
     \cos\angle(\vec x, \vec y) = \frac{\scpr{\vec x, \vec y}}{\norm{\vec x}\norm{\vec y}}.
 \end{equation*}
 In \cref{fig:biorth_vs_cond_cgs_o} we compare the condition number growth and the loss of biorthogonality for the deterministic method CGS\_O2 with the reciprocal of $\cos \angle(\vec q_i, \vec p_i)$, for $i = 1, \dots, 200$; similarly, for the randomized method rCGS\_O2 we compare against the reciprocal of $\cos \angle(\sketch \vec q_i, \sketch \vec p_i)$. Note that when $\cos \angle(\vec q_i, \vec p_i) \approx 0$, then $\vec q_i$ and $\vec p_i$ are almost orthogonal and we are in a near-breakdown situation for the two-sided Gram-Schmidt process; if this happens, we expect to have an increase in the condition number of the bases $\mat Q_i$ and $\mat P_i$ and in the loss of biorthogonality. Because of the discussion after \cref{prop:sketched-inner-product--gaussian}, we know that in the randomized setting the sketched vectors $\sketch \vec q_i$ and $\sketch \vec p_i$ are significantly less likely to be almost orthogonal to each other, so we expect to have a less severe growth of the condition number and of the loss of biorthogonality compared to the deterministic algorithms. Indeed, this expected behavior is observed in \cref{fig:biorth_vs_cond_cgs_o}: in particular, we see that the condition number and loss of biorthogonality have the most evident increases in correspondence of the highest spikes of $1/\cos \angle(\vec q_i, \vec p_i)$ (or $1 / \cos \angle(\sketch \vec q_i, \sketch \vec p_i)$ for the randomized algorithm), and that these spikes are several orders of magnitude smaller in the randomized setting.

\begin{figure}[tb]
	\centering
    \begin{subfigure}[b]{0.45\textwidth}
    \centering
    \includegraphics[width=\textwidth]{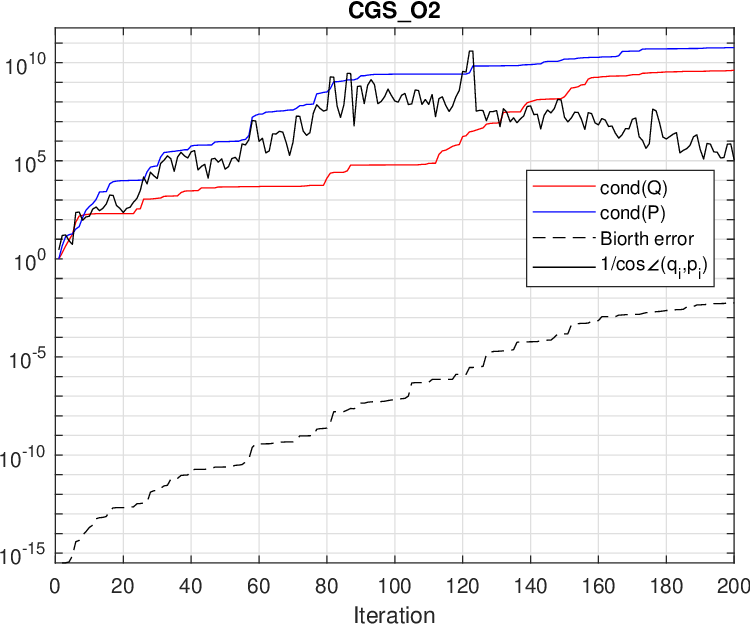}
    \caption{CGS\_O2}
    \label{fig:biorth_vs_cond_cgs_o2_ill}
    \end{subfigure}
    \hfill
    \begin{subfigure}[b]{0.45\textwidth}
    \centering
    \includegraphics[width=\textwidth]{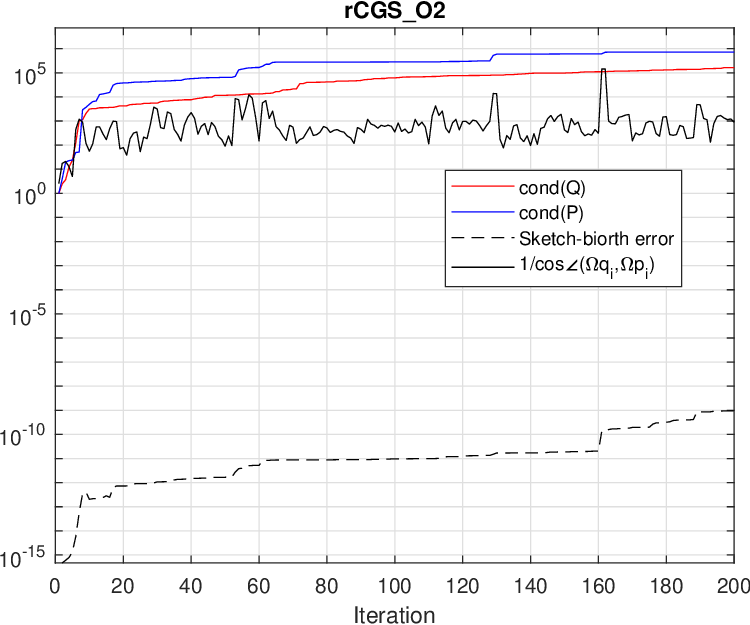}
    \caption{rCGS\_O2}
    \label{fig:biorth_vs_cond_rcgs_O2_ill}
    \end{subfigure}
    \caption{Comparison of the condition numbers of $\mat Q_i$ and $\mat P_i$, the loss of (sketch-)biorthogonality and the angle between the (sketches of) $\vec q_i$ and $\vec p_i$, for $\mat X$ and $\mat Y$ as in \cref{tab:exp1}. (a): CGS\_O2. (b): rCGS\_O2.}
	\label{fig:biorth_vs_cond_cgs_o}
\end{figure}

\subsubsection{Well-conditioned matrices}
\label{subsubsec:experiment-well-conditioned}

We now perform the same experiment as in \cref{subsubsec:experiment-ill-conditioned} using matrices with a significantly smaller condition number and a larger number of columns. We construct two random Gaussian matrices $\mat X,\mat Y\in \R^{n\times m}$ with i.i.d.~$\mathcal{N}(0,1)$ entries, with $n=10^4$ and $m=500$. In our experiment, the condition number of the two generated matrices are $\kappa(\mat X) = 1.572$ and $\kappa(\mat Y) = 1.564$. It is well-known that, even when starting with very well-conditioned matrices, the two-sided Gram-Schmidt process can still be unstable and construct highly ill-conditioned bases $\mat Q$ and $\mat P$.

\begin{table}
	\caption{Performance of two-sided Gram-Schmidt and randomized two-sided Gram-Schmidt on well-conditioned $\mat X$ and $\mat Y \in \R^{n \times m}$, with $n = 10^4$ and $m = 500$. 
	\label{tab:exp2}}
\centering
\begin{tabular}{l|cccccc}
\toprule
 & time (s) & ${\rm cond}(\mat Q)$ & ${\rm cond}(\mat P)$ & err $\mat X$ & err $\mat Y$ & (sketch-)biorth\\
\midrule
MGS & 3.410 & 2.197e+06 & 1.510e+06 & 3.676e-09 & 3.655e-09 & 4.631e-07 \\

MGS2 & 6.538 & 2.197e+06 & 1.510e+06 & 5.368e-09 & 3.886e-09 & 3.056e-10 \\

CGS & 1.421 & 9.417e+08 & 7.501e+08 & 6.986e-06 & 6.678e-06 & 2.916e+03 \\

CGS2 & 2.931 & 2.197e+06 & 1.510e+06 & 2.658e-09 & 2.625e-09 & 2.942e-10 \\

CGS3 & 3.771 & 2.197e+06 & 1.510e+06 & 2.684e-09 & 2.693e-09 & 3.239e-10 \\

CGS\_O & 3.236 & 2.197e+06 & 1.510e+06 & 2.606e-09 & 2.577e-09 & 1.341e-07 \\

CGS\_O2 & 5.122 & 2.197e+06 & 1.510e+06 & 2.650e-09 & 2.721e-09 & 3.512e-10 \\

\midrule

rMGS & 2.455 & 4.702e+05 & 3.565e+05 & 1.531e-09 & 1.509e-09 & 5.601e-08 \\

rMGS2 & 4.666 & 4.702e+05 & 3.565e+05 & 1.680e-09 & 1.699e-09 & 7.899e-11 \\

rCGS & 1.142 & 2.779e+06 & 2.639e+06 & 1.547e-08 & 1.428e-08 & 2.213e+03 \\

rCGS2 & 1.991 & 4.702e+05 & 3.565e+05 & 1.132e-09 & 1.123e-09 & 7.296e-11 \\

rCGS3 & 3.620 & 4.702e+05 & 3.565e+05 & 1.133e-09 & 1.122e-09 & 7.234e-11 \\

rCGS\_O & 2.654 & 4.702e+05 & 3.565e+05 & 7.957e-10 & 7.663e-10 & 1.926e-08 \\

rCGS\_O2 & 3.344 & 4.702e+05 & 3.565e+05 & 1.115e-09 & 1.108e-09 & 7.418e-11 \\
\bottomrule
    \end{tabular}

\end{table}

In \cref{tab:exp2} we report the results of all the variants of the deterministic and randomized two-sided Gram-Schmidt process.
Similarly to the experiment in \cref{subsec:experiment-condition-number-growth}, the randomized algorithms perform better than the deterministic ones. Even if the reduction in the basis condition number is less remarkable compared to \cref{tab:exp1}, the randomized algorithms still have smaller loss of biorthogonality and require significantly less computational time. For this example, rCGS2 is the fastest method among the ones that reach a sketch-biorthogonalization error around $10^{-10}$. 

In \cref{fig:biorth_vs_cond_cgs_o_exp2} we focus on CGS\_O2 and rCGS\_O2, and we compare the loss of biorthogonality with the condition number growth and the reciprocal of $\cos\angle(\vec q_i, \vec p_i)$, or the reciprocal of $\cos \angle(\sketch \vec q_i, \sketch \vec p_i)$ for the randomized algorithms. In this case, the different between the standard and the randomized approach is less significant than in \cref{subsubsec:experiment-ill-conditioned}, because even with the deterministic algorithm it is quite unlikely for $\vec q_i$ and $\vec p_i$ to be almost orthogonal. Nevertheless, we can still observe that the condition numbers of $\mat Q_i$ and $\mat P_i$ have particularly large growth spikes at the iterations in which $1/\cos\angle(\vec q_i, \vec p_i)$ or $1/\cos\angle(\sketch \vec q_i, \sketch \vec p_i)$ is large. We also notice that for the deterministic algorithm $1/\cos \angle(\vec q_i, \vec p_i)$ has slightly larger spikes, sometimes almost of the order $10^{7}$, than $1/\cos \angle(\sketch \vec q_i, \sketch \vec p_i)$ for the randomized algorithm, which is order $10^5$, with a ratio between one or two orders of magnitude.

\begin{figure}[htb]
    \centering
    \begin{subfigure}[b]{0.45\textwidth}
        \centering
        \includegraphics[width=\textwidth]{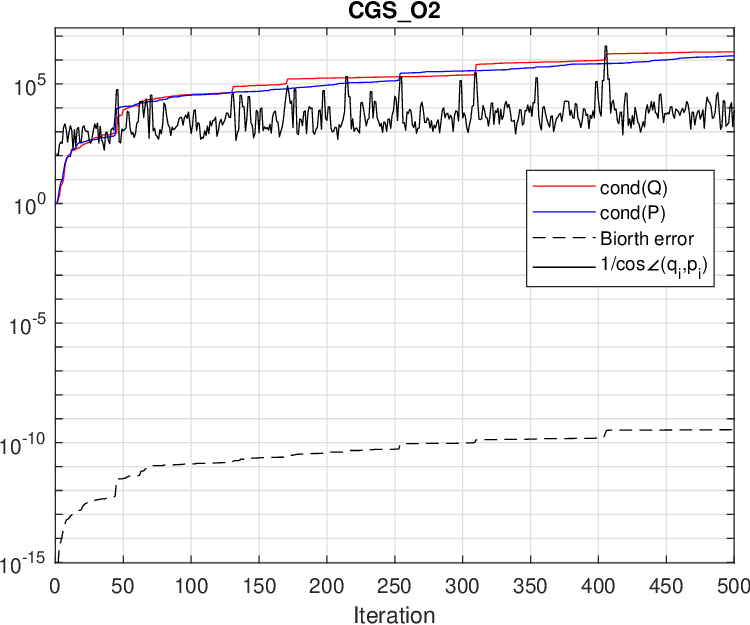}
        \caption{CGS\_O2}
        \label{fig:biorth_vs_cond_cgs2}
    \end{subfigure}
    \hfill
    \begin{subfigure}[b]{0.45\textwidth}
        \centering
        \includegraphics[width=\textwidth]{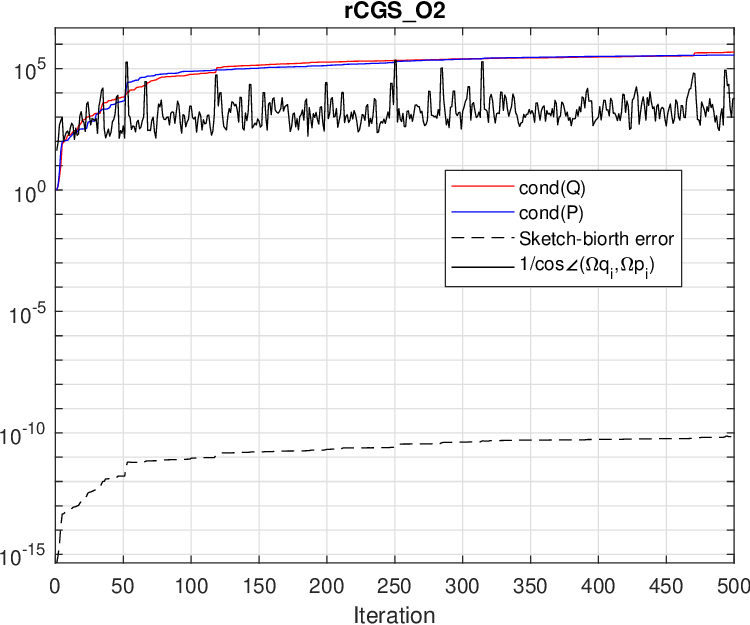}
        \caption{rCGS\_O2}
        \label{fig:biorth_vs_cond_rcgs2}
    \end{subfigure}
    \caption{Comparison of the condition numbers of $\mat Q_i$ and $\mat P_i$, the loss of (sketch-)biorthogonality and the angle between the (sketches of) $\vec q_i$ and $\vec p_i$, for $\mat X$ and $\mat Y$ as in \cref{tab:exp2}. (a) CGS\_O2. (b) rCGS\_O2.}
    \label{fig:biorth_vs_cond_cgs_o_exp2}
\end{figure}

\subsection{Application to nonsymmetric Lanczos for computing eigenvalues}\label{subsec:nonsymm_numerical_example}

In this section we apply the randomized two-sided Gram-Schmidt process as a sketch-biorthogonalization procedure within the randomized nonsymmetric Lanczos algorithm, introduced in \cref{subsec:randomized-nonsym-lanczos}. We use this algorithm to approximate a few leading eigenvalues of a matrix $\mat A \in \R^{n \times n}$, comparing it with the classical nonsymmetric Lanczos method with full biorthogonalization.  We construct a nonsymmetric matrix $\mat A\in\bb{R}^{n\times n}$ of size $n = 1000$ with prescribed eigenvalue distribution and condition number; more precisely, we take $\mat A = \mat X^{-1} \mat D \mat X$, where $\cond(\mat X) = 10^2$ and $\mat D$ is a diagonal matrix with eigenvalues $\lambda_1, \dots, \lambda_n$ given by
\begin{alignat*}{2}
		\lambda_i &= (0.95)^{i} & \qquad \text{for }i = 1, \dots, 15, \\
		\lambda_i &= (0.99)^{i-15} \lambda_{15} & \qquad \text{for } i = 16, \dots, n.
\end{alignat*} 
We run both methods for $m = 100$ iterations to compute the $10$ largest eigenvalues of $\mat A$.
Note that for this example, the nonsymmetric Lanczos algorithm implemented with a short-term recurrence found multiple copies of the leading eigenvalues and was unable to converge for the remaining ones, so full biorthogonalization is needed to ensure convergence.

We use MGS2 to biorthogonalize the bases in the deterministic Lanczos method, and CGS\_O2 for the randomized one, since they were the methods with the best compromise between performance and stability in \cref{subsec:experiment-condition-number-growth}. The convergence of the residuals of the right eigenvectors for the two methods are shown in \cref{fig:non_symm_lanczos}. We can see that the randomized Lanczos method exhibits almost the same convergence as the deterministic one, while employing a significantly cheaper sketched biorthogonalization procedure.

\begin{figure}[htb]
\centering
\begin{subfigure}[b]{0.45\textwidth}
    \centering
    \includegraphics[width=\textwidth]{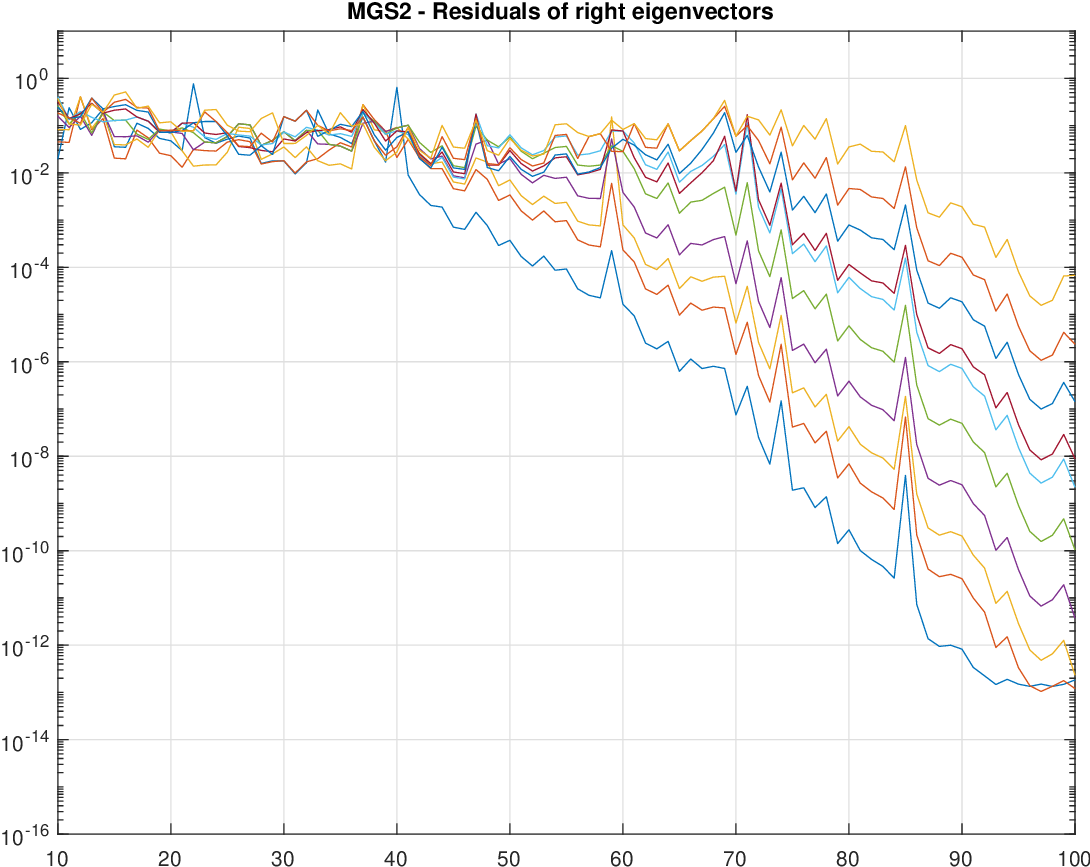}
    \caption{}
    \label{fig:sub_left}
\end{subfigure}
\hfill
\begin{subfigure}[b]{0.45\textwidth}
    \centering
    \includegraphics[width=\textwidth]{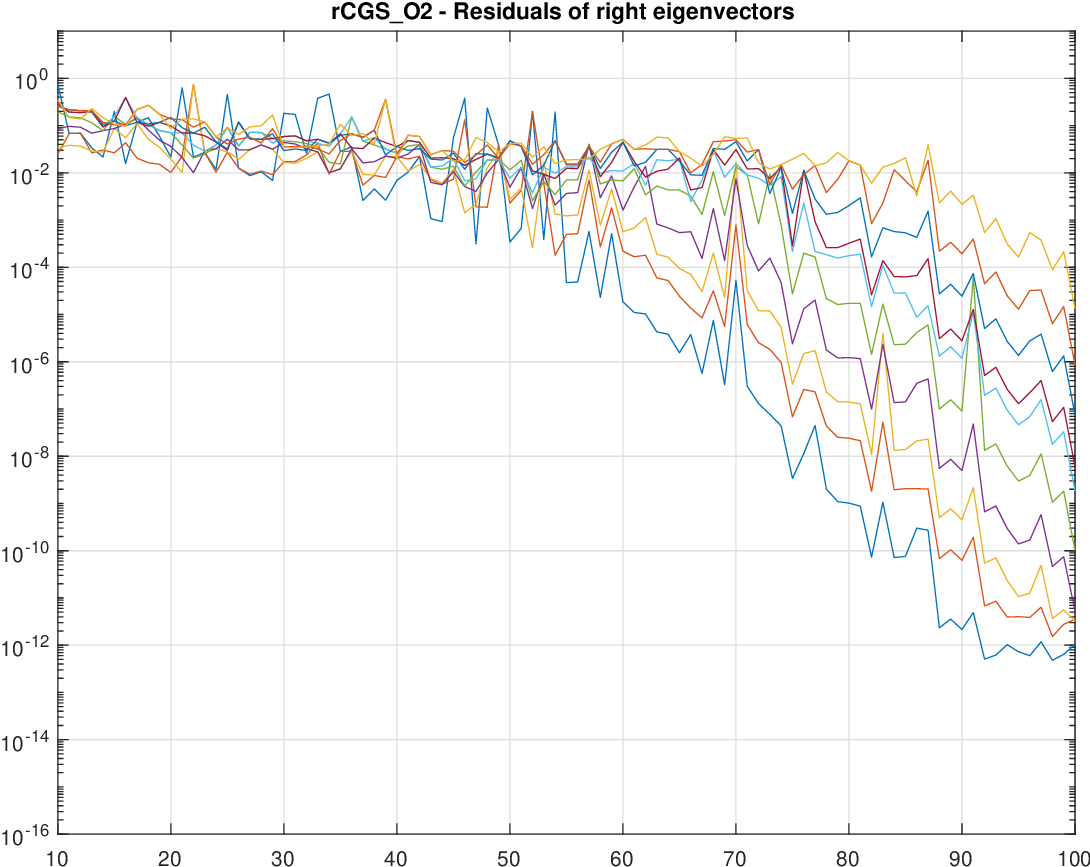}
    \caption{}
    \label{fig:sub_right}
\end{subfigure}
\caption{Convergence of the eigenvectors associated with the 10 largest eigenvalues of the matrix $\mat A$ defined in \cref{subsec:nonsymm_numerical_example}. (a) Nonsymmetric Lanczos with MGS2. (b) Randomized nonsymmetric Lanczos with rCGS\_O2.
\label{fig:non_symm_lanczos}}
\end{figure}

\section{Concluding remarks}
\label{sec:conclusions}

In this paper we have presented a randomized two-sided Gram-Schmidt algorithm, which constructs two bases that are sketch-biorthogonal instead of exactly biorthogonal. Since the length of the sketched vectors is significantly smaller than the length of the original vectors, the randomized algorithm has a significantly lower computation time compared to its deterministic counterpart. 
In addition to the increased computational efficiency, we have shown that the randomized two-sided Gram-Schmidt process has improved numerical stability and constructs better conditioned sketch-biorthogonal bases, especially in situations in which the biorthogonal bases constructed by the standard two-sided Gram-Schmidt process are severely ill-conditioned. This phenomenon is justified by the fact that the basis vectors constructed by the randomized algorithm are unlikely to be almost sketch-orthogonal, which helps in avoiding the ``worst-case scenario'' in which the condition number of the basis grows dramatically.
We have proposed and compared several different implementations of the sketched oblique projector in order to identify the best compromise between efficiency and numerical stability. According to our numerical experiments, the best overall implementation in terms of computational performance and stability is rCGS\_O2.
We have also proposed a nonsymmetric Lanczos algorithm which employs the randomized two-sided Gram-Schmidt algorithm to compute sketch-biorthogonal Krylov bases, and we have demonstrated its competitive performance against the standard nonsymmetric Lanczos algorithm for the computation of the leading left and right eigenvectors of a nonsymmetric matrix.

\section*{Acknowledgments}
We are grateful for insightful discussions with Yousef Saad and Valeria Simoncini. 
The first author has received funding from the European Research Council (ERC) under the European Union’s Horizon 2020 research and innovation program (grant agreement No 810367). The work of the second author is funded by the European Union under the National Recovery and
Resilience Plan (PNRR) - Mission 4 - Component 2 Investment 1.4 “Strengthening research structures and creation of “National R\&D Champions” on some Key Enabling Technologies” DD N. 3138 of 12/16/2021
rectified with DD N. 3175 of 18/12/2021, code CN00000013 - CUP J33C22001170001. 

\bibliographystyle{siam}
\bibliography{paper-biblio}

\newpage

\appendix

\section{A technical lemma}

The following technical lemma is needed in the proof of \cref{prop:sketched-inner-product--gaussian}. 
\begin{lemma}
	\label{lemma:k0-bound-near-zero}
	Let $K_0$ be the modified Bessel function of the second kind of order $0$. For all $u \in (0, 1)$, we have
	\begin{equation}
		\label{eqn:k0-bound-near-zero}
		K_0(u) \le -\ln u + K_0(1).
	\end{equation} 
\end{lemma}
\begin{proof}
	For completeness, we recall that the modified Bessel function of the second kind of order $0$ can be defined as (see \cite[eq.~(9.6.24)]{AbramowitzStegun64})
	\begin{equation*}
		\Kzero(u) := \int_{0}^{\infty}e^{-u\cosh t}\,dt, \qquad u > 0.
	\end{equation*} 
	However, this representation will not be needed for this proof.
	Define $g(u) := K_0(u) + \ln u$. If we show that $g'(u) > 0$ for all $u \in (0,1)$, we get $g(u) \le g(1) = K_0(1)$, which is equivalent to \cref{eqn:k0-bound-near-zero}. We have
	\begin{equation*}
		g'(u) = K_0'(u) + \frac{1}{u} = -K_1(u) + \frac{1}{u},
	\end{equation*}   
	where we used $K_0'(u) = K_1(u)$ (see \cite[eq.~(9.6.27)]{AbramowitzStegun64}). To prove that $g'(u) > 0$, it is therefore enough to show that $h(u) := u K_1(u) < 1$ for all $u \in (0,1)$. 
	Using \cite[eq.~(9.6.26)]{AbramowitzStegun64}, we have
	\begin{equation*}
		K_1'(u) = -\frac{1}{2}(K_0(u) + K_2(u)) \qquad \text{and} \qquad K_0(u) - K_2(u) = -\frac{2}{u} K_1(u),
	\end{equation*}
	so we obtain
	\begin{equation*}
		K_1'(u) = -K_0(u) - \frac{1}{u} K_1(u) \qquad \text{and} \qquad h'(u) = -u K_0(u).
	\end{equation*}
	Since $K_0(u) > 0$ for $u > 0$, we conclude that $h'(u) < 0$ for all $u \in (0,1)$, and therefore $h(u) \le \lim_{u \to 0^+} h(u) = 1$, which follows from the asympototic expansion $K_1(u) \approx \frac{1}{u}$ for $u \to 0^+$ (see \cite[eq.~(9.6.9)]{AbramowitzStegun64}). This shows in turn that $g'(u) > 0$ for all $u \in (0, 1)$, which concludes the proof.      
\end{proof}

\section{Mixed precision implementation}

In this section we investigate the behavior of \cref{algorithm:two-sided-gram-schmidt-sketched} when implemented using mixed precision. The advantages of using mixed precision in different applications have been extensively studied in modern literature, see e.g.~\cite{HighamMary22, YangFoxSanders21, YamazakiTomovDongarra15}. 
In particular, we are interested in the variants of the Gram-Schmidt process using multiprecision finite arithmetic \cite{YTKDB15, OktayCarson23, BalabanovGrigori22}. We choose to use {\tt single}, or {\tt fp32}, for low precision, and {\tt double}, or {\tt fp64}, for high precision. High precision is used for operations with sketched quantities, that is, application of the sketching operator to vectors, the updates of sketched vectors, and the biorthogonalization steps for the sketched bases. All other operations are performed in lower precision.

In \cref{tab:mix_comparison}, we compare the mixed precision implementations against the previous ones in \texttt{double} precision, on the ill-conditioned example presented in \cref{subsec:experiment-condition-number-growth}. We only report the results for the algorithms that converge both in \texttt{double} and in mixed precision. As expected, the algorithms using mixed precision are faster than their standard counterparts, especially rCGS3 and rCGS\_O2. Nonetheless, the loss of accuracy in the decompositions of $\mat X$ and $\mat Y$ is significantly higher than in \texttt{double} precision; note that the ratio between the errors in the mixed precision implementation and the one in \texttt{double} precision is approximately nine orders of magnitude, which is slightly higher than the ratio between the unit roundoffs between \texttt{single} and \texttt{double} precision. 
On an easy test problem, for example when $\mat X \approx \mat Y$, so that the two-sided Gram-Schmidt process is almost equivalent to the Gram-Schmidt orthogonalization of the matrix $\mat X$ and $\cond(\mat Q) \approx \cond(\mat P) \approx 1$, the mixed precision implementations achieve decomposition errors $\err(\mat X)$ and $\err(\mat Y)$ of order $10^{-8}$. Hence, we expect that the large decomposition errors in \cref{tab:mix_comparison} are mainly due to the high condition number of the computed bases $\mat Q$ and $\mat P$.

\begin{table}
	\caption{Comparison of the \texttt{double} precision and mixed precision implementations of randomized two-sided Gram-Schmidt, on the example in \cref{subsubsec:experiment-ill-conditioned}. The mixed precision algorithms have the prefix ``mp-'' in the table.
    \label{tab:mix_comparison}}
\centering
\begin{tabular}{l|cccccc}
\toprule
 & time (s) & $\cond(\mat Q)$ & $\cond(\mat P)$ & $\err(\mat X)$ & $\err(\mat Y)$ & sketch-biorth\\
\midrule
rMGS2 & 0.659 & 1.333e+05 & 5.699e+05 & 5.999e-12 & 3.980e-11 & 2.527e-11 \\

rCGS3 & 0.419 & 3.107e+05 & 9.504e+05 & 5.215e-12 & 3.308e-11 & 3.050e-11 \\

rCGS\_O2 & 0.320 & 1.639e+05 & 7.254e+05 & 4.943e-12 & 3.259e-11 & 9.432e-10 \\

\midrule 

mp-rMGS2 & 0.575 & 2.326e+06 & 8.905e+06 & 2.836e-03 & 1.943e-02 & 2.592e-11 \\

mp-rCGS3 & 0.192 & 2.325e+06 & 7.455e+06 & 2.223e-03 & 1.498e-02 & 2.741e-11 \\

mp-rCGS\_O2 & 0.185 & 2.326e+06 & 7.459e+06 & 2.224e-03 & 1.498e-02 & 2.804e-11 \\
\bottomrule
    \end{tabular}
\end{table}

\end{document}